\documentclass{article}
\usepackage{amsmath,amsthm,mathrsfs}
\usepackage{latexsym}
\usepackage{amssymb}
\usepackage{enumerate}

\usepackage{hyperref}%
\hypersetup{
	urlcolor = red,
	colorlinks = true,
	linkcolor = blue,
	citecolor = blue,
	linktocpage = true,
	pdftitle = {Diffeomorphisms on simple polytopes and controllability},
	pdfauthor = {Helge Gl\"{o}ckner, Erlend Grong, Alexander Schmeding},
	bookmarksopen = true,
	bookmarksopenlevel = 1,
	unicode = true,
	hypertexnames =false
}

\usepackage{cleveref} %For referencing

\usepackage[
backend=biber,
style=alphabetic,
maxnames=6,
sorting=nyt,
giveninits=true,
]{biblatex}
\newtheorem{thm}{Theorem}[section]
\newtheorem{la}[thm]{Lemma}
\newtheorem{Defn}[thm]{Definition}
\newtheorem{Remark}[thm]{Remark}
\newtheorem{Conj}[thm]{Conjecture}
\newtheorem{prop}[thm]{Proposition}
\newtheorem{cor}[thm]{Corollary}
\newtheorem{Example}[thm]{Example}
\newtheorem{Examples}[thm]{Examples}
\newenvironment{defn}{\begin{Defn}\rm}{\end{Defn}}
\newenvironment{example}{\begin{Example}\rm}{\end{Example}}
\newenvironment{examples}{\begin{Examples}\rm}{\end{Examples}}
\newenvironment{rem}{\begin{Remark}\rm}{\end{Remark}}

\newtheorem{Number}[thm]{\!\!}
\newenvironment{numba}{\begin{Number}\rm}{\end{Number}}

\newcommand{\cV}{{\mathcal V}}

\newcommand{\cE}{{\mathcal E}}
\newcommand{\cL}{{\mathcal L}}

\newcommand{\cA}{{\mathcal A}}
\newcommand{\cF}{{\mathcal F}}

\newcommand{\scrV}{{\mathscr V}}
\newcommand{\scrU}{{\mathscr U}}
\newcommand{\scrG}{{\mathscr G}}

\DeclareMathOperator{\spn}{span}
\newcommand{\ve}{\varepsilon}
\newcommand{\R}{{\mathbb R}}
\newcommand{\N}{{\mathbb N}}

\newcommand{\mto}{\mapsto}
\newcommand{\sub}{\subseteq}

\DeclareMathOperator{\id}{id}
\DeclareMathOperator{\intOP}{int}
\DeclareMathOperator{\Spann}{span}

\DeclareMathOperator{\Gr}{Gr}

\DeclareMathOperator{\Ext}{Ext}
\DeclareMathOperator{\Evol}{Evol}

\DeclareMathOperator{\Fl}{Fl}
\DeclareMathOperator{\ev}{ev}

\newcommand{\cg}{{\mathfrak g}}

\DeclareMathOperator{\Supp}{supp}
\DeclareMathOperator{\aff}{aff}
\DeclareMathOperator{\ind}{ind}
\DeclareMathOperator{\algint}{algint}
\DeclareMathOperator{\str}{str}
\DeclareMathOperator{\conv}{conv}
\DeclareMathOperator{\Diff}{Diff}
\DeclareMathOperator{\Vect}{Vect}
\DeclareMathOperator{\fr}{fr}
\DeclareMathOperator{\im}{im}

\begin{document}
\begin{center}
{\bf\Large Boundary values of diffeomorphisms\vspace{2mm}
of simple polytopes, and controllability}\\[4.8mm]
{\bf Helge Gl\"{o}ckner, Erlend Grong and Alexander Schmeding}
\end{center}
\begin{abstract}
%\hspace*{-4.8mm}
\noindent We consider the Lie group of smooth diffeomorphisms $\Diff(M)$ of a simple polytope $M$ in the euclidean space. Simple polytopes are special cases of manifolds with corners. The geometric setting allows to study in particular, the subgroup of face respecting diffeomorphisms and its Lie theoretic properties. We find a canonical Lie group structure for the quotient of the diffeomorphism by the subgroup $\Diff^{\partial,\id}(M)$ of maps that equal the identity on the boundary, turning the canonical quotient homomorphism
$\Diff(M)\to\Diff(M)/\Diff^{\partial,\id}(M)$
into a smooth submersion. We also show that the identity component of the diffeomorphism group is generated by the exponential image, by proving general controllability results.
\vspace{3mm}
\end{abstract}
\textbf{MSC 2020 subject classification:}
58D05 (primary);
% Groups of diffeomorphisms and homeomorphisms as manifolds 
%
22E65,
% inf-dim Lie groups
%
34H05,
% Control problems involving ordinary differential equations
%
52B08,
% Combinatorial properties of polytopes and polyhedra
%
52B70,
% polyhedral manifolds
%
53C17
% Sub-Riemannian geometry
(secondary)\\[2.3mm]
\textbf{Keywords:}
Simple polytope, manifold with corners,
diffeomorphism group,
quotient, controllability, extension operator.

\tableofcontents

\section{Introduction}
We obtain results concerning diffeomorphism
groups of polytopes,
which are always assumed to be \emph{convex} polytopes.
Given a polytope $M\sub\R^n$
with non-empty interior,
the group $\Diff(M)$ of $C^\infty$-diffeomorphisms
$\phi\colon M\to M$
can be considered as a Lie group~\cite{Glo23}.
As a consequence, $\Diff(M)$ is a Lie group
for any polytope~$M$ in a finite-dimensional
vector space (see \ref{DiffP}).
The diffeomorphisms which are \emph{face-respecting}
in the sense that $\phi(F)=F$ for each face $F$ of~$M$
form an open normal subgroup $\Diff^{{\fr}}(M)$
of $\Diff(M)$ of finite index.
We obtain results for an important
class of polytopes, the \emph{simple}
polytopes (see \cite[\S12]{Bro83}, or also \S5 in \cite[Chapter 6.5]{Bar02}).
An $n$-dimensional
polytope
is \emph{simple}
if each vertex is contained in precisely $n$ edges
of~$M$ (see Definition~\ref{defn-sim} for this
and other known characterizations
of simplicity).
We show that a polytope is simple
if and only if it can be
regarded as a smooth manifold with corners
in the sense of~\cite{Cer61, Dou61, Mic80}
(see Definition~\ref{defnculi}, \ref{poly-rough},
and Proposition~\ref{main-equivalence}).
Each polytope in $\R^1$ or $\R^2$
is simple, as well as all cubes $[0,1]^n$,
all simplices (like the tetrahedron)
and the dodecahedron (see Examples~\ref{exa-simp}).
Our first main result is the following.
\begin{thm}\label{thm-A}
Let $M$ be a simple polytope
of dimension $n\geq 2$ and let $\ell \in \{1,\ldots, n-1\}$.
Let $\cF$ be the set of all faces of~$M$
of dimension~$\ell$.
Then the image $\im(\rho)$ of the group homomorphism
\[
\rho\colon \Diff^{\fr}(M)\to\prod_{F\in \cF}\Diff^{\fr}(F),\quad
\phi\mto (\phi|_F)_{F\in\cF}
\]
is a submanifold of the direct product
$\prod_{F\in \cF}\Diff^{\fr}(F)$ and hence a Fr\'{e}chet--Lie group.
The submanifold structure
turns $\rho\colon \Diff^{\fr}(M)\to\im(\rho)$
into a smooth submersion.\footnote{In the sense of \cite[Definition~4.4.8]{Ham82}.}
In particular, the latter map admits smooth local sections. If $\ell=1$, then
$\cF$ is the set of all edges of~$M$ and
\[
\rho(\Diff(M)_0)=\prod_{F\in \cF}\Diff(F)_0
\]
holds for the connected components of the identity.
\end{thm}
\noindent
The Fr\'{e}chet--Lie group $\im(\rho)$ is $L^1$-regular in the sense of \cite{Glo15,Nik21}
and hence a regular Lie group in Milnor's sense~\cite{Mil84}
(see Proposition~\ref{is-regular}).
\begin{rem}\label{post-thm-B}
Taking $\ell=n-1$, the kernel of~$\rho$ is the Lie subgroup
$\Diff^{\partial,\id}(M)$ of $\Diff(M)$
consisting of all diffeomorphisms
which fix the boundary pointwise.
In particular, the theorem shows that $\Diff^{\fr}(M)/\Diff^{\partial,\id}(M)$
can be made a regular Fr\'{e}chet--Lie group in such a way that
\[
\Diff^{\fr}(M)\to \Diff^{\fr}(M)/\Diff^{\partial,\id}(M)
\]
is a smooth $\Diff^{\partial,\id}(M)$-principal bundle,\footnote{In the sense
of~\cite[Definition 3.7.23]{GN25}.}
using the right multiplication of $\Diff^{\partial,\id}(M)$
on $\Diff^{\fr}(M)$.
\end{rem}
\begin{cor}\label{new-cor}
For each simple polytope~$M$
of dimension $\geq 2$,
there exists a regular Fr\'{e}chet--Lie group structure
on $\Diff(M)/\Diff^{\partial,\id}(M)$
which turns the canonical quotient homomorphism
\[
\Diff(M)\to \Diff(M)/\Diff^{\partial,\id}(M)
\]
into a smooth submersion
and makes $\Diff(M)$ a smooth $\Diff^{\partial,\id}(M)$-principal bundle
using the right multiplication of $\Diff^{\partial,\id}(M)$
on $\Diff(M)$.
\end{cor}
\noindent
The proof of Theorem~\ref{thm-A} uses
a result concerning
continuous linear extension\linebreak
operators
for compatible vector-valued $C^m$-functions on faces of polytopes,
which entails
an extension result for the relevant
vector fields (Corollary~\ref{vecfield-cor}).

\begin{thm}\label{thm-B}
Let $M$ be a simple polytope
of dimension $n\geq 1$. Let $\ell\in \{1,\ldots, n-1\}$
and $\cF$ be the set of all faces of~$M$ of dimension~$\ell$.
Let $Y$ be a locally convex topological vector space
and $m\in \N_0\cup\{\infty\}$.
Let $\cE$ be the closed vector subspace
of $\prod_{F\in \cF}C^m(F,Y)$
consisting of all $(f_F)_{F\in \cF}$ such that
\[
f_F|_{F\cap G}=f_G|_{F\cap G}\quad\mbox{for all $F,G\in \cF$.}
\]
Then the mapping
\[
r\colon C^m(M,Y)\to \cE,\quad f\mto (f|_F)_{F\in \cF}
\]
is surjective, continuous, linear, and admits a continuous linear
right inverse $\sigma\colon \cE\to C^m(M,Y)$.
\end{thm}
\noindent
We mention that, for $\ell=1$ and $Y\not=\{0\}$,
the conclusion of Theorem~\ref{thm-B}
becomes false for each $n$-polytope~$M$
which is not simple. Likewise, the conclusion of
Corollary~\ref{vecfield-cor} becomes false
if $\ell=1$, for each non-simple $n$-polytope~$M$
See Remark~\ref{converse} for details.\\[2mm]
We also study controllability on simple
polytopes, where our result can be summarised as follows. For a simple polytope $M$, let $\Vect_{\mathrm{str}}(M)$ be the space of (smooth) vector fields $X$ on $M$ such that if $x$ is contained in a face $F$, then $X(x) \in T_x F$ where the face is a manifold with corners as in \Cref{mfd-corners}. The principal part of such a vector field is a stratified vector field as defined in \Cref{facesandstratvec} (cf.\ \cite[Remark 5.7]{Glo23}). Since the stratified vector fields form the Lie algebra of the group of face-respecting diffeomorphisms $\Diff^{\fr}(M)$, we exploit that this group is regular in the sense of Milnor, \cite{Glo23}. The evolution of smooth Lie algebra valued curves exist and are smooth. One can show (see Section \ref{sect:control} for details) that the flow of the smooth time dependent vector field $\hat{X}(t) := tX \in \Vect_{\mathrm{str}}(M)$ is a one parameter curve of time-dependent diffeomorphisms
$$\varphi_t = e^{tX} =\Evol (\hat{X})(t).$$
Here $\Evol$ is the evolution map for the regular Lie group $\Diff^{\fr} (M)$. In particular, $\varphi_t$ is a face preserving diffeomorphism of $M$ for each fixed $t$. To avoid confusion, let us stress that since $\Diff^{\fr}(M)$ is an open subgroup of $\Diff(M)$, the connected component of the identity in $\Diff^{\fr}(M)$ coincides with the one of $\Diff(M)$. We will thus suppress the distinction in the following to shorten the notation.

For a subspace $\scrV$ of $\Vect_{\mathrm{str}}(M)$, the subgroup of $\Diff(M)_0$ generated by it is denoted by
$$\Gr(\scrV) = \left\{ e^{X_1} \circ \cdots \circ e^{X_k} \, : \, X_1, \dots, X_k \in \scrV, k=0,1,2,\dots \right\} $$
We further write $\hat \scrV$ for the $C^\infty(M)$-module generated by $\scrV$
$$\hat \scrV = \spn_{C^\infty(M)} \scrV =  \spn_{\R}\left\{ fX \, :\, f\in C^\infty(M), X \in \scrV \right\}.$$
With this notation we have the following result.
\begin{thm}\label{thm-control}
Let $M$ be a simple polytope of dimension $n \geq 1$, and define $\scrV$ as a subspace of $\Vect_{\mathrm{str}}(M)$. Assume that the following conditions are satisfied.
\begin{enumerate}[\rm (I)]
\item $\Gr(\scrV)$ acts transitively on the interior of $M$ and on the interior of each face $F$.
\item For any face $F$, if we we write it as an intersection of facets $F =\cap_{j=1}^i \hat F_j$ then for every $j=1,\dots,i$, there exists a point $x_j \in F$, neighbourhood $V_j$ of $x_j$ in $M$ and a vector field $Z_j \in \hat \scrV $ satisfying
$$Z_j|_{V_j \cap \hat F_j} =0, \qquad \nabla_{\nu_j} Z_j(x_j) \neq T_{x_j}\hat F_j$$
relative to any covariant derivative $\nabla$ on $M$ and any local vector field $\nu_j$ on $M$ satisfying $\nu_j(x) \neq T_{x_j} \hat F_j$.
\end{enumerate}
Then we have
$$\Gr(\hat \scrV) = \Diff(M)_0.$$
\end{thm}
Note that condition (II) in \Cref{thm-control} does not depend on the choice of covariant derivative. Setting $\scrV = \Vect_{\mathrm{str}}(M)$, we get the following corollary of Theorem~\ref{thm-control}.
\begin{cor}\label{cor-connected-component}
For each simple polytope~$M$
of dimension $n\geq 1$,
the identity component $\Diff(M)_0$
of $\Diff(M)$
is generated by the exponential image.
\end{cor}
\noindent
This result mirrors the result for a manifold with boundary, see \cite{Luk78} for a sketch,
with a detailed proof found in \cite{GS24}. Moreover as in \cite{GS24}, we have:
\begin{thm}\label{thm-identity-boundary}
For each simple polytope~$M$
of dimension $n\geq 1$,
the group $\Diff^{\partial,\id}(M)$ of diffeomorphisms that are the identity on the boundary are generated by the exponential image of vector fields $\Vect^{\partial=0}(M)$ vanishing on the boundary.
\end{thm}

The controllability arguments are local in nature, so we conjecture that these results carry over to well behaved manifolds with corners. A simple polytope which is a manifold with corners is automatically a so called manifolds with faces, \cite{Jae68} and cf. \Cref{rem:mf_faces}. Manifolds with faces require more boundary regularity and rule out problematic boundary intersections. Hence results from \Cref{thm-control} and \Cref{thm-identity-boundary} should carry over to the more general setting of manifolds with faces without problems. However, currently there is no complete description of diffeomorphism groups of manifolds with corners (or faces) as (infinite-dimensional) Lie groups. We refer to \cite{Glo23} for more information.
\section{Preliminaries}
All vector spaces we consider are vector spaces
over~$\R$.
Hausdorff, locally convex
topological vector spaces
will simply be called locally convex spaces
and their topology a locally convex vector topology.
If $E$ is a vector space and $S\sub E$,
we let $\Spann(S)$ denote the vector subspace of~$E$
spanned by~$S$.
If $X$ is a topological space and $Y\sub X$
a subset, we shall frequently say that a subset
$U\sub Y$ is \emph{open in $Y$}
or an \emph{open subset of~$Y$}
if $U$ is relatively open.
In the following, $\N$ denotes the set of positive integers,
while $\N_0:=\N\cup\{0\}$.\\[2.3mm]
We now compile basic concepts
and notation concerning
polytopes (cf.\ \cite{Bro83}).
\begin{numba}
A \emph{polytope}
in a finite-dimensional vector space~$E$
is the convex hull $\conv(S)$
of a non-empty finite subset $S\sub E$.
Each polytope is compact (see \cite[Theorem~7.1]{Bro83}).
If $M\sub E$ is a polytope,
we let $\aff(M)$ be the affine
subspace of~$E$ generated by~$M$.
If $n$ is the dimension of $\aff(M)$,
we call $n$ the \emph{dimension of~$M$}
and say that $M$ is an \emph{$n$-polytope}.
The \emph{algebraic interior}
$\algint(M)$ of~$M$ is its interior
as a subset of $\aff(M)$
(called the \emph{relative interior} in \cite{Bro83}).
The algebraic interior is dense in~$M$
and $\aff(M)=\aff(\algint(M))$
(cf.\ (c) and (f) in \cite[Theorem~3.4]{Bro83}).
A convex subset $F\sub M$ is called
a \emph{face} of~$M$ if
\[
(\forall x,y\in M)\, (\forall t\in \, ]0,1[)\;\;\,
tx+(1-t)y\in F\;\Rightarrow\; x,y\in F.
\]
Each non-empty face $F$ of~$M$
is a polytope and the number of faces of~$M$
is finite (see \cite[Theorem~7.3
and Corollary~7.4]{Bro83}).
If $N$ is a face of~$M$,
then a subset $F\sub N$ is a face of~$N$
if and only if it is a face of~$M$
(cf.\ \cite[Theorem~5.2]{Bro83}).
If $F$ is not empty and $d$ its dimension,
we call $F$ a \emph{$d$-face} of~$M$.
Points $x_0\in M$ such that $\{x_0\}$
is a face of~$M$ are called \emph{vertices};
$1$-faces of~$M$ are called \emph{edges} of~$M$.
If $M$ is an $n$-polytope with $n\geq 1$,
then its $(n-1)$-faces are called
\emph{facets} of~$M$.
Each $x\in M$ is contained in a smallest face $M(x)$
of~$M$, the intersection of all faces
containing $x$. If $d$ is its dimension, we call
\[
\ind_M(x):=n-d
\]
the \emph{index} of~$x$ in~$M$.
The sets $\algint(F)$
form a partition of~$M$
for $F$ in the set of non-empty faces of~$M$
(see \cite[Corollary~5.7]{Bro83}).
For a face $F$ of~$M$,
we have $F=M(x)$ if and only if $x\in \algint(F)$.
If $M$ is an $n$-polytope
and $i\in \{0,1,\ldots, n\}$,
we define $\partial_i(M):=\{x\in M\colon \ind_M(x)=i\}$.
\end{numba}
\begin{example}\label{ex:square}
For the square $S:=[0,1]^2$, we have
$\partial_0(S)=\;]0,1[^2$,
$\partial_1(S)=(]0,1[\;\times\{0,1\})\cup(\{0,1\}\times \;]0,1[)$
and $\partial_2(S)=\{0,1\}\times\{0,1\}$.
The vertices of $S$ are $(0,0)$, $(0,1)$, $(1,0)$ and $(0,1)$,
the edges are $\{0\}\times [0,1]$, $\{1\}\times[0,1]$,
$[0,1]\times\{0\}$ and $[0,1]\times\{1\}$
and these are also the four faces of~$S$.
\end{example}
\noindent
We use $C^m$-maps
between open subsets of
locally convex spaces
in the sense of Bastiani~\cite{Bas}
and the corresponding $C^\infty$-manifolds
and Lie groups
(see \cite{Res,GN25,Nee06,Sch23}
for further information;
cf.\ also \cite{Ham82,Mic80,Mil84,Nee06,Sch23}).
Thus, manifolds
and Lie groups
are modeled
on locally convex spaces
which can be infinite dimensional,
unless the contrary is stated.
For
the basic definition,
consider
locally
convex spaces $E$ and $F$, an open subset $U\sub E$
and a mapping $f\colon U\to F$.
We write
\[
D_yf(x):=\frac{d}{dt}\Big|_{t=0}f(x+ty)=\lim_{t\to0}\frac{1}{t}(f(x+ty)-f(x))
\]
for the directional derivative of
$f$ at $x\in U$ in the direction $y\in E$,
if it exists.
If $f$ is continuous,
we also say that $f$ is $C^0$
and write $d^0f:=f$.
\begin{numba}
Given $m\in \N\cup\{\infty\}$,
a function $f\colon U\to F$
is called a \emph{$C^m$-map}
if it is continuous
and, for each $k\in \N$,
the iterated directional derivative
\[
d^kf(x,y_1,\ldots, y_k):=(D_{y_k}\cdots D_{y_1}f)(x)
\]
exists for each $k\in \N$ with $k\leq m$
and all $x\in U$ and $y_1,\ldots, x_k\in E$,
and the mappings
\[
d^kf\colon U\times E^k\to F
\]
so obtained are continuous.
We abbreviate $df:=d^1f$.
As usual, $C^\infty$-maps
are also called \emph{smooth}.
\end{numba}
\begin{numba}
If $G$ is a Lie group modeled
on a locally convex space,
with neutral element~$e$,
we write $L(G):=T_eG$
for its Lie algebra.
If $\alpha\colon G\to H$
is a smooth group homomorphism
between Lie groups, we let $L(\alpha):=T_e(\alpha)\colon L(G)\to L(H)$
be the associated continuous Lie algebra
homomorphism.
\end{numba}
\begin{numba}
If $G$ is a Lie group,
then for each $v\in L(G)$
there is at most one smooth group
homomorphism $\gamma_v\colon (\R,+)\to G$
such that
$\dot{\gamma}_v(0)=v$.
If $\gamma_v$ exists for each $v\in L(G)$,
then
$\exp_G\colon L(G)\to G$, $v\mto\gamma_v(1)$
is called the exponential function.
% of~$G$.
\end{numba}
\begin{numba}
If $M$ is a smooth manifold modeled
on a locally convex space~$E$
and $F$ a closed vector subspace of~$E$,
then a subset $N\sub M$ is called
a \emph{submanifold} of~$M$
modeled on~$F$ if, for each $x\in N$,
there exists a chart $\phi\colon U\to V\sub E$
of~$M$ with $x\in U$ such that $\phi(U\cap N)=V\cap F$.
Endowing~$N$ with the induced topology,
the maximal $C^\infty$-atlas containing
the restrictions $\phi|_{U\cap N}\colon U \cap N\to V\cap F$
makes~$N$ a smooth manifold modeled on~$F$.
\end{numba}
\begin{numba}
Let $M$ and $N$ be smooth manifolds
modeled on locally convex spaces~$E$ and $F$,
respectively.
A smooth function $q\colon M\to N$
is called a \emph{smooth submersion}
if, for each $x\in M$,
there exist charts $\phi\colon U_\phi\to V_\phi\sub E$
of~$M$
and $\psi\colon U_\psi\to V_\psi\sub F$
of~$N$
and a continuous linear mapping $\alpha\colon E\to F$
admitting a continuous linear right inverse
such that
$x\in U_\phi$, $y\in U_\psi$, $f(U_\phi)\sub U_\psi$,
and $\psi\circ f|_{U_\phi}\circ \phi^{-1}=\alpha|_{V_\phi}$.
\end{numba}
\noindent
We frequently
need to consider differentiable functions
on non-open sets, like relatively
open subsets of $[0,\infty[^n$.
We follow an approach described
in \cite{GN25}.
\begin{numba}
A subset $U$ of $\R^n$ (or a locally convex space~$E$)
is called \emph{locally convex}
if each $x\in U$ has a relatively open
neighbourhood $V$ in~$U$
which is convex. Then each $x$-neighbourhood in~$U$
contains such a neighbourhood.
As a consequence, relatively open subsets
of locally convex sets are locally convex.
Each convex subset $U$ of~$E$
is locally convex.
\end{numba}
\begin{numba}
Let $E$ and $F$ be locally
convex spaces, $U\sub E$ be a locally convex subset
with dense interior $\mathring{U}$ and
$m\in \N_0\cup\{\infty\}$.
A function $f\colon U\to F$ is called $C^m$
if $f$ is continuous, the restriction
$f|_{\mathring{U}}$ is $C^m$ and the iterated directional
derivatives $d^k(f|_{\mathring{U}})\colon \mathring{U}\times E^k\to F$
have continuous extensions $d^kf\colon U\times E^k\to F$
for all $k\in \N$ such that $k\leq m$
(see \cite[Definition 1.4.4 and Lemma~1.4.5]{GN25}).
Again, $C^\infty$-maps are also called \emph{smooth}.
\end{numba}
\begin{numba}
The set $C^m(U,F)$
of all $C^m$-functions $f\colon U\to F$
is a vector space with pointwise operations.
We endow $C^m(U,F)$ with the
compact-open $C^m$-topology, i.e.,
the initial topology with respect to the linear mappings
\[
C^m(U,F)\to C(U\times E^k,F),\quad f\mto d^kf
\]
for $k\in \N_0$ with $k\leq m$,
using the compact-open topology
on the vector space $C(U\times E^k,F)$
of all continuous functions $U\times E^k\to F$
(see \cite[Definition 1.7.10]{GN25}).
\end{numba}
\noindent
The following fact (see \cite[Proposition 1.7.11]{GN25})
will be used repeatedly.
\begin{numba}\label{basicpull}
Let $E_1$, $E_2$, and $F$ be locally
convex spaces, $U_1\sub E_1$ and $U_2\sub E_2$
be locally convex subsets with dense interior,
$m\in \N_0\cup\{\infty\}$
and $\psi\colon U_1\to U_2$ be a $C^m$-map.
Then
\[
\psi^*\colon C^m(U_2,F)\to C^m(U_1,F),\quad f\mto f\circ\psi
\]
is a continuous linear map.
\end{numba}
\begin{numba}\label{spec}
If $X$ is a topological space and $W\sub X$ a subset
with dense interior, then each relatively
open subset of~$W$ has dense interior.
Notably,
if $A$ is a convex subset of~$\R^n$
with non-empty interior
(e.g., $A=[0,\infty[^n$),
then each relatively open subset $V$ of $A$
is a locally convex subset of~$\R^n$ with dense
interior,
enabling us to speak about $C^m$-functions on~$V$.
\end{numba}
\begin{rem}
If the convex set $A\sub \R^n$ with non-empty interior
is closed in~$\R^n$,
then alternative characterizations
are possible for $C^m$-maps from a relatively open subset $V\sub A$
to a locally
convex space~$F$
(which shall not be used in the following,
but link our framework to
other popular approaches):\\[2.3mm]
Write $V=U\cap A$ for
an open subset $U$ of~$\R^n$.
If $m\in \N_0$ or
$F$ is metrizable or $F$ is sequentially
complete or $A=[0,\infty[^n$,
then a function $f\colon V\to F$
is $C^m$ in the above sense
if and only if it extends to a $C^m$-function
$U\to F$
(see \cite[Korollar 3.2]{Jak23},
\cite[Korollar 3.10]{Jak23},
\cite[Theorem 1.10\,(b)]{Glo22},
and \cite[Theorem 1.10\,(c)]{Glo22}, respectively).
If $m\in \N_0$ or $F$ is sequentially
complete or $A=[0,\infty[^n$,
then the references actually provide
a continuous linear
extension operator $C^m(V,F)\to C^m(U,F)$,
exploiting suitable
versions of Seeley's extension theorem (as
in \cite{Han23}) and Whitney's
extension theorem (as in \cite{RS21,Jak23}).
For finite-dimensional~$F$,
this is classical (cf.\ \cite{Whi34}
and \cite[Theorem~2.3]{Bie80}).
\end{rem}
\section{Extension of compatible {\boldmath$C^\infty$}-functions
to half-open cubes}\label{sec-loc}
We prepare the proof
of Theorem~\ref{thm-B}.
\begin{numba}\label{setth}
Given $i\in \N$
and $k\in \{1,\ldots, i\}$,
we consider the subset
\[
F_{i,k}:=\{x\in [0,1[^i\colon x_k=0\}
\]
of $[0,1[^i$, writing $x=(x_1,\ldots, x_i)$.
Let $Q$ be a locally convex subset with dense
interior of a locally convex space~$Z$
and $m\in \N_0\cup\{\infty\}$.
Since $F_{i,k}\times Q$ is
a locally convex set with non-empty interior
in the vector space $\{x\in \R^i\colon x_k=0\}\times Z$,
we can speak about $C^m$-functions
on $F_{i,k}\times Q$.
Given a locally convex topological
vector space~$Y$,
consider the closed vector subspace
$E_i$ of
\[
\prod_{k=1}^iC^m(F_{i,k}\times Q,Y)
\]
containing all $(f_1,\ldots, f_i)\in \prod_{k=1}^iC^m(F_{i,k}\times Q,Y)$
such that
\[
f_k(x,q)=f_\ell(x,q)
\]
for all $k,\ell\in\{1,\ldots, i\}$,
$x\in F_{i,k}\cap F_{i,\ell}$, and $q\in Q$.
We give $C^m(F_{i,k}\times Q,Y)$ the compact-open $C^m$-topology
for each $k$, as well as $C^m([0,1[^i\times Q,Y)$.
On $E_i$, we consider
the topology induced by the product.
Note that the continuous linear mapping
\[
\rho_i\colon C^m([0,1[^i\times Q,Y)\to \prod_{k=1}^i C^m(F_{i,k}\times Q,Y),\quad
f\mto (f|_{F_{n,k}\times Q})_{k=1}^i
\]
takes values in $E_i$;
we use the same symbol, $\rho_i$,
for its co-restriction to a map $C^m([0,1[^i\times Q,Y)\to E_i$.
\end{numba}
\noindent
We now describe a continuous linear right inverse
$\Phi_i\colon E_i\to C^m([0,1[^i\times Q,Y)$
for the map $\rho_i\colon C^m([0,1[^i\times Q,Y)\to E_i$.
\begin{numba}\label{reunot}
It would be enough for us to consider,
for $n\in \N$ with $n\geq i$,
the case that $Z:=\R^{n-i}$
and $Q:=\;]{-1},1[^{n-i}$
(with $\;]{-1},1[^0:=\{0\}$). We also write
\[
E_{n,i}:=E_n\sub \prod_{k=1}^i C^m(F_{i,k}\times \;]{-1},1[^{n-i},Y)
\]
in this case, $\Phi_{n,i}:=\Phi_i\colon E_{n,i}\to C^m([0,1[^i\times \;]{-1},1[^{n-i},Y)$
and
\[
\rho_{n,i}:=\rho_i\colon C^m([0,1[^i\times \;]{-1},1[^{n-i},Y)\to E_{n,i}.
\]
If $Q$ is a singleton, we may omit $Q$
and constant variables in $Q$
in the notation.
% when convenient.
\end{numba}
\begin{numba}
For $j\in \{1,\ldots, i\}$,
we write $[i,j]$ for the set of all
subsets $S\sub \{1,\ldots,i\}$
having exactly~$j$ elements.
For $S\in [i,j]$,
abbreviate $k(S):=\min(S)$.
We let
\[
\theta_{i,S}\colon [0,1[^i\times Q\to F_{i,k(S)}\times Q
\]
be the mapping taking $(x_1,\ldots, x_i,q)\in [0,1[^i\times Q$
to the element $(y_1,\ldots,y_i,q)\in F_{i,k(S)}\times Q$
with components
\[
y_\ell=\left\{
\begin{array}{cl}
0 & \mbox{\,if $\,\ell\in S$;}\\
x_\ell & \mbox{\,if $\,\ell\not\in S$}
\end{array}\right.
\]
for $\ell\in \{1,\ldots, i\}$.
Then $\theta_{i,S}$
is the restriction of a continuous linear mapping $\R^i\times Z\to\R^i\times Z$
and thus~$C^\infty$.
% smooth
By \ref{basicpull},
the linear map
\[
(\theta_{i,S})^*\colon C^m(F_{i,k(S)}\times Q,Y)\to C^m([0,1[^i\times Q,Y),\quad
f\mto f\circ \theta_{i,S}
\]
is continuous. Hence
$\Psi_i\colon \!\prod_{k=1}^i C^m(F_{i,k}\times Q,Y)\to C^m([0,1[^i\times Q,Y)$,
\[
(f_k)_{k=1}^i\mto
\sum_{j=1}^i(-1)^{j-1}\!\!\sum_{S\in [i,j]}\! f_{k(S)}\circ \, \theta_{i,S}
\]
is a continuous linear map, and so is its restriction
\begin{equation}\label{thePhi}
\Phi_i:=\Psi_i|_{E_i}\colon E_i\to C^m([0,1[^i\times Q,Y).
\end{equation}
Thus, for all $(f_1,\ldots, f_i)\in E_i$ and
$x=(x_1,\ldots, x_i,q)\in [0,1[^i\times Q$, we have
\begin{equation}\label{evaluated}
\Phi_i(f_1,\ldots, f_i)(x,q)=
\sum_{j=1}^i(-1)^{j-1}\!\!\sum_{S\in [i,j]}\! f_{k(S)}(\theta_{i,S}(x,q)).
\end{equation}
\end{numba}
\begin{rem}
Explicitly, for $Q=Z=\{0\}$ we get
\begin{eqnarray}
\Phi_1(f_1)(x_1)&=&f_1(0),\label{easiest}\\
\Phi_2(f_1,f_2)(x_1,x_2)&=& f_1(0,x_2)+f_2(x_1,0)-f_1(0,0),\notag\\
\Phi_3(f_1,f_2,f_3)(x_1,x_2,x_3)&=&
f_2(0,x_2,x_3)+f_1(x_1,0,x_3)+f_1(x_1,x_2,0)\notag\\
& & - f_1(0,0,x_3)-f_1(0,x_2,0)-f_2(x_1,0,0)\notag\\
& &+f_1(0,0,0)\notag
\end{eqnarray}
and
\begin{eqnarray*}
\lefteqn{\Phi_4(f_1,f_2,f_3,f_4)(x_1,x_2,x_3,x_4)}\\
\!\!\!\!\!\!\!&=& f_1(0,x_2,x_3,x_4)+ f_2(x_1,0,x_3,x_4)+f_3(x_1,x_2,0,x_4)
+ f_4(x_1,x_2,x_3,0)\\
\!\!\!\!\!\!\!&&
- f_1(0,0,x_3,x_4) - f_1(0,x_2,0,x_4)-f_1(0,x_2,x_3,0)
-f_2(x_2,0,0,x_4)\\
\!\!\!\!\!\!\! && \quad -f_2(x_1,0,x_3,0)
-f_3(x_1,x_2,0,0)\\
&&
+ f_1(0,0,0,x_4) + f_1(0,0,x_3,0) + f_1(0,x_2,0,0)+f_2(x_1,0,0,0)\\
&&- f_1(0,0,0,0)
\end{eqnarray*}
in the cases $i\in \{1,2,3,4\}$, for
$(f_k)_{k=1}^i\in E_i$ and $(x_1,\ldots,x_i)\in [0,1[^i$.
For arbitrary $Z$ and $Q$, add a parameter $q\in Q$ in each term.
\end{rem}
\begin{prop}\label{local-ext-vertex}
The continuous linear map
$\Phi_i\colon E_i\to C^m([0,1[^i\times Q,Y)$
defined in {\rm(\ref{thePhi})}
is a right inverse for
$\rho_i\colon C^m([0,1[^i\times Q,Y)\to E_i$, $f\mto (f|_{F_{i,k}})_{k=1}^i$.
\end{prop}
\begin{proof}
If $i\geq 2$ and $k\in \{1,\ldots, i-1\}$,
then
\[
\eta_{k,x_i}(x_1,\ldots, x_{i-1},q):=(x_1,\ldots, x_i,q)\in F_{i,k}\times Q
\]
for
all $(x_1,\ldots,x_{i-1},q)\in F_{i-1,k}\times Q$
and $x_i\in [0,1[$,
furnishing a smooth map $\eta_{k,x_i}\colon F_{i-1,k}\times Q\to F_{i,k}\times Q$.
For $k$ and $x_i$ as before,
we therefore obtain a $C^m$-function
\[
f(\cdot,x_i,\cdot):=f\circ \eta_{k,x_i}\in C^m(F_{i-1,k}\times Q,Y),\quad
(x_1,\ldots,x_{i-1},q)\mto f(x_1,\ldots, x_i,q)
\]
for all $f\in C^m(F_{i,k}\times Q,Y)$.
Evaluating for $k\not=\ell$
at elements $(x_1,\ldots,x_{i-1},q)\in (F_{n-1,k}\cap F_{i-1,\ell})\times Q$,
we check that
\[
(f_1(\cdot,x_i,\cdot),\ldots,f_{n-1}(\cdot,x_i,\cdot))\in E_{i-1}
\quad\mbox{for all $\,(f_1,\ldots, f_i)\in E_i$.}
\]
We now prove the assertion of the proposition
by induction on $i\in \N$,
and that
\begin{eqnarray}
\lefteqn{\Phi_i(f_1,\ldots, f_i)(x_1,\ldots,x_i,q)}\notag\\
&=&\Phi_{i-1}(f_1(\cdot,x_i,\cdot),\ldots, f_{i-1}(\cdot,x_i,\cdot))(x_1,\ldots, x_{i-1},q)
+f_i(x_1,\ldots, x_{i-1},0,q)\notag\\
&& -\Phi_{i-1}(f_1(\cdot,0,\cdot),\ldots, f_{i-1}(\cdot,0,\cdot))(x_1,\ldots, x_{i-1},q)
\label{for-induc}
\end{eqnarray}
holds if $i\geq 2$, for all $(f_1,\ldots, f_i)\in E_i$
and $(x_1,\ldots, x_i,q)\in [0,1[^i\times Q$.\\[2.3mm]
The case $i=1$: Then $F_{1,1}=\{0\}$.
For each $f_1\in E_1=C^m(\{0\}\times Q,Y)$,
the map $\Phi_1(f_1)$ is given by
\[
\Phi_1(f_1)(x_1,q)=f_1(0,q)\quad\mbox{for all $x_1\in [0,1[$ and $q\in Q$,}
\]
(cf.\ (\ref{easiest})),
whence $\Phi_1(f_1)(0,q)=f_1(0,q)$.
As a consequence, $\Phi_1(f_1)|_{F_{1,1}\times Q}=f_1$
and hence $\rho_1(\Phi_1(f_1))=f_1$.\\[2.3mm]
Let $i\geq 2$ now and assume
the assertion holds for $i-1$ in place of~$i$.
It suffices to show~(\ref{for-induc}).
In fact, if (\ref{for-induc})
holds, we infer that
\[
\Phi_i(f_1,\ldots,f_i)|_{F_{i,k}\times Q}=f_k
\]
for all $(f_1,\ldots, f_i)\in E_i$
and $k\in \{1,\ldots, i\}$ (whence $\rho_i\circ \Phi_i=\id_{E_i}$),
as follows: Let $(x,q)=(x_1,\ldots, x_i,q)\in F_{i,k}\times Q$.
If $k=i$, then $x=(x_1,\ldots, x_{i-1},0)$,
whence the first and final summand in~(\ref{for-induc})
cancel and we get
\[
\Phi_i(f_1,\ldots, f_i)(x,q)=f_i(x_1,\ldots,x_{i-1},0,q)=f_i(x,q).
\]
If $k<i$, then
$(x_1,\ldots, x_{i-1})\in F_{i-1,k}$, whence the first
summand in (\ref{for-induc}) is
\begin{eqnarray}
\lefteqn{\Phi_{i-1}(f_1(\cdot,x_i,\cdot),\ldots, f_{i-1}(\cdot,x_i,\cdot))
(x_1,\ldots, x_{i-1},q)}
\qquad\quad\notag\\
&=&f_k(\cdot,x_i,\cdot)(x_1,\ldots, x_{i-1},q)
= f_k(x_1,\ldots,x_i,q)=f_k(x,q).\label{to-repeat}
\end{eqnarray}
Since
$(x_1,\ldots, x_{i-1},0)\in F_{i,k}\cap F_{i,i}$,
we have
\begin{equation}\label{end-zero}
f_k(x_1,\ldots, x_{i-1},0,q)=f_i(x_1,\ldots, x_{i-1},0,q),
\end{equation}
by definition of~$E_i$.
Repeating the calculation of (\ref{to-repeat})
with $(x_1,\ldots, x_{i-1},0)$ in place of~$x$,
we see that the final summand in
(\ref{for-induc}) is given by
\begin{eqnarray*}
\lefteqn{-\Phi_{i-1}(f_1(\cdot,0,\cdot),\ldots, f_{i-1}(\cdot,0,\cdot))(x_1,\ldots, x_{i-1},q)}\qquad\quad
\\
&=& -f_k(x_1,\ldots,x_{i-1},0,q)=-f_i(x_1,\ldots, x_{i-1},0,q),
\end{eqnarray*}
using (\ref{end-zero}) for the final equality.
The final summand in~(\ref{for-induc}) therefore cancels the penultimate
summand; using (\ref{to-repeat}) to re-write the first
summand, we obtain
\[
\Phi_i(f_1,\ldots, f_i)(x,q)=\Phi_{i-1}(f_1(\cdot,x_i,\cdot),\ldots, f_{i-1}
(\cdot, x_i,\cdot))(x_1,\ldots, x_{i-1},q)=f_k(x,q).
\]
Thus $\Phi_i(f_1,\ldots, f_i)|_{F_{i,k}\times Q}=f_k$ also for $k<i$.\\[2.3mm]
To get~(\ref{for-induc}), write $\theta_{i-1,S}$
for the map $C^m([0,1[^{i-1}\times Q,Y)\to C^m(F_{i-1,k(S)}\times Q,Y)$
for subsets $S\sub \{1,\ldots,i-1\}$
which is obtained by replacing $i$ with $i-1$
in the definition of~$\theta_{i,S}$.
Let $x=(x_1,\ldots, x_i)\in [0,1[^i$
and $q\in Q$.
For $j=1$ and $S=\{i\}$,
the summand in (\ref{evaluated}) is
\[
(-1)^{j-1}f_{k(S)}(\theta_{i,S}(x,q))=f_i(x_1,\ldots,x_{i-1},0,q)
\]
and hence equals the second summand in (\ref{for-induc}).\\[2.3mm]
For $j\in \{1,\ldots,i-1\}$ and $S\in [i,j]$
with $i\not\in S$, we have $S\in [i-1,j]$
and the corresponding summand in (\ref{evaluated})
is
\[
(-1)^{j-1}f_{k(S)}(\theta_{i,S}(x,q))=(-1)^{j-1}
(f_{k(S)}(\cdot,x_i,\cdot)\circ\theta_{i-1,S})(x_1,\ldots, x_{i-1},q).
\]
The sum of these terms over all $(j,S)$
yields
\[
\Phi_{i-1}(f_1(\cdot,x_i,\cdot),\ldots, f_i(\cdot,x_i,\cdot))(x_1,\ldots,x_{i-1},q),
\]
which is the first summand in~(\ref{for-induc}).\\[2.3mm]
For $j\in \{2,\ldots, i\}$,
the sets $S\in [i,j]$ with $i\in S$
are in bijection with sets $S'\in [i-1,j-1]$
via $S\mto S':=S\setminus\{i\}$, $S'\mto S:=S'\cup\{i\}$.
We have $k(S)=\min S=\min S'=k(S')$.
The summand in (\ref{evaluated}) corresponding~to~$(j,S)$~is
\begin{eqnarray}
\lefteqn{(-1)^{j-1}f_{k(S)}(\theta_{i,S}(x_1,\ldots, x_i,q))}\qquad\qquad\notag\\
& = &  -(-1)^{(j-1)-1}f_{k(S')}(\cdot,0,\cdot)(\theta_{i-1,S'}(x_1,\ldots,x_{i-1},q)).
\label{to-be-summed}
\end{eqnarray}
This is the negative of the summand for $(j-1,S')$
in place of $(j,S)$ in the formula for
\[
\Phi_{i-1}(f_1(\cdot,0,\cdot),\ldots,f_{i-1}(\cdot,0,\cdot))(x_1,\ldots,x_{i-1},q)
\]
analogous to~(\ref{evaluated}).
The sum of the term in (\ref{to-be-summed})
over all $(j,S)$ therefore equals
\[
-\Phi_{i-1}(f_1(\cdot,0,\cdot),\ldots,f_{i-1}(\cdot,0,\cdot))(x_1,\ldots,x_{i-1},q),
\]
the final summand in~(\ref{for-induc}).
Thus~(\ref{for-induc}) holds, which completes the proof.
\end{proof}
\begin{rem}\label{rem-in-span}
The formula (\ref{evaluated})
implies that
\begin{equation}\label{in-span}
\Phi_i(f)([0,1[^i\times Q)\sub \Spann \bigcup_{k=1}^i f_k(F_{i,k}\times Q)
\end{equation}
for each
$f=(f_1,\ldots, f_i)\in E_i$,
in the situation of Proposition~\ref{local-ext-vertex}.
\end{rem}
\section{Polytopes which are manifolds with corners}
We show that a polytope is simple
if and only if it can be regarded as a manifold
with corners.
First, we recall the classical concept
of a simple polytope.
\begin{defn}\label{defn-sim}
Let $n\in \N_0$. An $n$-polytope~$M$
is called \emph{simple} if
the following
equivalent conditions are
satisfied:
\begin{itemize}
\item[(a)]
Each vertex of~$M$ is contained in exactly
$n$ edges of~$M$.
\item[(b)]
Each vertex of~$M$ is contained in exactly
$n$ facets of~$M$.
\item[(c)]
For each $k\in \{0,\ldots, n-1\}$
and $k$-face $F$ of~$M$,
the number of facets of~$M$
containing~$F$ equals $n-k$.
\end{itemize}
\end{defn}
\begin{rem}
Condition~(c) is used as the definition in \cite[p.\,80]{Bro83};
for the equivalences, see \cite[Theorems~12.11 and 12.12]{Bro83}.
A vertex of any $n$-polytope
is contained in at least $n$ edges of~$M$
(see \cite[Theorem~10.5]{Bro83}).
\end{rem}
\begin{examples}\label{exa-simp}
(a) For each $n\in \N$, the cube $[0,1]^n$ is a
simple $n$-polytope.\smallskip

\noindent
(b) Each $2$-polytope in $\R^2$
is simple, and likewise each $1$-polytope in~$\R$.\smallskip

\noindent
(c) The tetrahedron
is simple.
More generally, each simplex is simple.\smallskip

\noindent
(d) The dodecahedron is simple.\smallskip

\noindent
(e) If $M$ and $N$ are simple polytopes,
then also $M\times N$ is simple.\footnote{If $M$ is an $m$-polytope
and $N$ an $n$-polytope, then $M\times N$
has dimension $m+n$. Each vertex of $M\times N$
is of the form $(x_0,y_0)$ with vertices $x_0$ of~$M$
and $y_0$ of~$N$.
The edges of $M\times N$
containing $(x_0,y_0)$ are $\{x_0\}\times F$
with $F$ an edge of~$N$ containing~$y_0$,
and $E\times \{y_0\}$ with $E$ an edge of~$M$
containing $x_0$. There are $m+n$ of these.}\smallskip

\noindent
(f) Each non-empty face of a simple polytope is simple \cite[Theorem~12.15]{Bro83}.\smallskip

\noindent
(g) Chopping off a vertex of a simple
polytope, the resulting truncated polytope
is simple (see \cite[Theorem~12.18]{Bro83}).
\end{examples}
\begin{rem}
The icosahedron is not simple.
If $A\sub \R^2$ is a $2$-polytope
with $\geq 4$
vertices, then a pyramid in~$\R^3$
with base~$A$ is a non-simple $3$-polytope.
\end{rem}
\begin{numba}
Each $n$-polytope
$M\sub\R^n$ is a locally convex subset of
$\R^n$ with dense interior, whence
each relatively open subset $U\sub M$
is a locally convex subset of $\R^n$ with dense
interior. For a relatively open subset $V$
in $[0,\infty[^n$ (or in $[0,\infty[^i\times\R^{n-i}$
for some $i\in \{0,\ldots,n\}$),
we can therefore call a map
$\phi\colon U\to V$ a $C^\infty$-diffeomorphism
if it is smooth to $\R^n$
and invertible with smooth inverse~$\phi^{-1}$.
\end{numba}
\begin{defn}\label{defnculi}
We say that an $n$-polytope
$M\sub\R^n$ is \emph{cube-like}
if $M$
is covered by the domains $U$ of $C^\infty$-diffeomorphisms
$\phi\colon U\to V$
from relatively open subsets $U\sub M$ onto
relatively open subsets $V\sub [0,\infty[^n$
(writing $[0,\infty[^0:=\{0\}$).
\end{defn}

More generally,
we shall call an $n$-polytope $P\sub E$
in a finite-dimensional vector space~$E$
cube-like if there exists an injective
affine map $f\colon \R^n\to E$ with $f(\R^n)=\aff(P)$
such that $f^{-1}(P)$ is a cube-like $n$-polytope in $\R^n$.
\begin{numba}\label{poly-rough}
Each cube-like $n$-polytope $M$ in $\R^n$
becomes an $n$-dimensional smooth manifold with corners
(in the sense recalled in Definition~\ref{mfd-corners})
if we endow it with the
maximal cornered $C^\infty$-atlas $\cA$
containing each $\phi$ as in Definition~\ref{defnculi}.
If $P$ is an $n$-polytope in a finite-dimensional
real vector space~$E$, we choose an affine
bijection $f\colon \R^n\to \aff(P)$
and transport the structure of
smooth manifold with corners from
$f^{-1}(P)$ to~$P$. Given a locally convex space~$Y$
and $m\in \N_0\cup\{\infty\}$,
we consider the $C^\infty$-diffeomorphism
$\psi\colon f^{-1}(P)\to P$, $x\mto f(x)$ 
and give $C^m(P,Y)$ the locally convex vector
topology making the bijective linear map
\[
\psi^*\colon C^m(P,Y)\to C^m(f^{-1}(P),Y),\quad g\mto g\circ\psi
\]
an isomorphism of topological vector spaces. As a consequence
of \ref{basicpull}, it is independent of the choice of~$f$.
\end{numba}
\noindent
For the reader's convenience, we recall (cf.\
\cite{Cer61, Dou61, Mic80}):
\begin{defn}\label{mfd-corners}
Let $n\in \N_0$.
An $n$-dimensional smooth manifold with corners
is a Hausdorff topological space~$M$,
together with a maximal set $\cA$
of homeomorphisms $\phi\colon U_\phi\to V_\phi$
from open subsets $U\sub M$ onto relatively open
subsets $V\sub [0,\infty[^n$,
such that $\bigcup_{\phi\in \cA}U_\phi=M$ holds
and the transition maps
$\phi\circ \psi^{-1}\colon \psi(U_\phi\cap U_\psi)\to
\phi(U_\phi\cap U_\psi)$ are smooth
for all $\phi,\psi\in \cA$.
The elements $\phi\in \cA$ are called the \emph{charts}
of~$M$ and $\cA$ a maximal cornered $C^\infty$-atlas.
Let $M$ and $N$ be smooth manifolds with corners.
A map $f\colon M\to N$ is called \emph{smooth}
if $f$ is continuous and $\phi\circ f\circ\psi^{-1}\colon \psi(U_\psi\cap f^{-1}(U_\phi))\to
V_\phi$
is smooth for each chart $\psi$ of~$M$ and
each chart $\phi$ of~$N$.
\end{defn}
\noindent
As cube-like polytopes are manifolds with corners, we have:
\begin{la}\label{sta-cha}
If $M$ is a cube-like $n$-polytope,
$x\in M$ and $i:=\ind_M(x)$, then there exists
an open $x$-neighbourhood $U\sub M$
and a $C^\infty$-diffeomorphism
\[
\kappa\colon U\to [0,1[^i\times \,]{-1},1[^{n-i}
\]
such that $\kappa(x)=0$ and
$U\cap F=\emptyset$ for all
facets $F$ of~$M$ with $x\not\in F$.
\end{la}
\noindent
We shall call such a diffeomorphism a \emph{standard chart}
of~$M$ around~$x$.
\begin{proof}
Since $M$ is cube-like,
there exists a $C^\infty$-diffeomorphism
$\phi\colon W\to V$
from
an open $x$-neighbourhood $W\sub M$
onto an open subset $V\sub [0,\infty[^n$.
As the union $A$ of facets not containing~$x$
is closed, after replacing $W$ with $W\setminus A$
we may assume that $W\cap A=\emptyset$.
Let $y:=(y_1,\ldots, y_n):=\phi(x)$.
After a permutation of the coordinates,
we may assume that $y_1=\cdots=y_i=0$
and $y_{i+1},\ldots, y_n>0$.
We may assume that $V\sub [0,\infty[^i\times \,]0,\infty[^{n-i}$,
after intersecting with the latter relatively open set.
For some $\ve>0$, we have $Q:=[0,\ve[^i\times \,]{-\ve},\ve[^{n-i}\sub V-y$.
Then $U:=\phi^{-1}(Q)$ is an open $x$-neighbourhood in~$M$ and
\[
\kappa\colon U\to [0,1[^i\times\,]{-1},1[^{n-i},\quad
z\mto \frac{1}{\ve}(\phi(z)-y)
\]
is as desired.
\end{proof}
%
%
%\begin{rem}
\noindent
Cube-like polytopes are useful
because
the standard charts just constructed
will enable the
extension results of Section~\ref{sec-loc}
to be applied locally.
We show:
\begin{prop}\label{main-equivalence}
The following properties
are equivalent for each polytope~$M$:
\begin{itemize}
\item[\rm(a)]
$M$ is a simple polytope.
\item[\rm(b)]
$M$ is cube-like.
\end{itemize}
\end{prop}
\begin{proof}
We may assume that $M$ is an $n$-polytope in~$\R^n$.
As singletons and compact intervals
are both simple and cube-like,
we may assume that $n\geq 2$.\medskip

``(a)$\Rightarrow$(b)'':
It suffices to show that,
for each $x\in M$, there exists
a $C^\infty$-diffeomorphism
$\kappa\colon U\to V$ from a relatively
open subset $U\sub M$ with $x\in U$
onto an open subset $V\sub \;[0,\infty[^i\times \R^{n-i}$
for some $i\in \{0,\ldots, n\}$.
For each $x$ in the interior of~$M$,
we can take $U=V=\mathring{M}$ equal to the interior, $i=0$ and $\kappa:=\id_{\mathring{M}}$.
Let $x\in M\setminus \mathring{M}$ now
and $M(x)$ be the face of~$M$
generated by~$x$. After a translation, we may assume that $x=0$.
Let $m$ be the number of facets of~$M$.
There exist
non-zero continuous linear functionals
$\lambda_1,\ldots,\lambda_m\in (\R^n)^*$
and real numbers $a_1,\ldots,a_m$
such that, setting
\[
K_j:=\{z\in \R^n\colon \lambda(z) \geq a_j\}
\]
for $j\in \{1,\ldots, m\}$, we have
\[
M=\bigcap_{j=1}^mK_j
\]
and
\[
F_j:=\{z\in \R^n\colon \lambda_j(z) =a_j\}\quad
\mbox{for $\,j\in \{1,\ldots,m\}$}
\]
are the facets of~$M$
(see \cite[Corollary~9.6]{Bro83}).
After a permutation,
we may assume that
$F_1,\ldots, F_i$ are the facets of~$M$
which contain $x$ (and hence also $F(x)$),
for some $i\in \{1,\ldots,m\}$.
Then
\[
W:=\{
z\in \R^n\colon (\forall j\in \{i+1,\ldots, m\})\;
\lambda_j(z)>a_j\}
\]
is an open subset of $\R^n$ such that $x\in W$ and
\[
W\cap M=W\cap \bigcap_{j=1}^i K_j.
\]
By \cite[Theorem~10.4]{Bro83},
$M(x)=\bigcap_{j=1}^iF_j$ holds.
Hence $M(x)$ is an $(n-i)$-face of~$M$,
by \cite[Theorem~12.14]{Bro83},
and thus $i=\ind_M(x)\in \{1,\ldots, n\}$.
Since $0=x\in F_j$ for $j\in \{1,\ldots, i\}$,
we must have $a_j=0$ for all $j\in \{1,\ldots,i\}$,
whence
\[
H_j:=\{z\in\R^n\colon \lambda_j(z)=0\}
\]
is a vector subspace of~$\R^n$ such that $F_j=H_j\cap M$.
Let $Y_k:=\bigcap_{j=1}^kH_k$ for $k\in \{1,\ldots, i\}$.
Now $\bigcap_{j=1}^kF_k=M\cap \bigcap_{j=1}^kH_j=M\cap Y_k$
is an $n-k$-face of~$M$ for $k\in \{1,\ldots,i\}$, by \cite[Theorem~12.14]{Bro83}.
As a consequence, $Y_{k+1}$ must be a proper vector subspace of $Y_k$
for all $k\in\{1,\ldots,i-1\}$.
Since each $Y_k$ has codimension $\leq k$ in~$\R^n$,
we deduce that $\dim(Y_k)=n-k$ for all $k\in \{1,\ldots,i\}$.
As a consequence, $\lambda_1,\ldots, \lambda_i$ are linearly independent.
Let $e_1,\ldots, e_n$ be the standard basis vectors
of $\R^n$ and $e_1^*,\ldots, e_n^*\in(\R^n)^*$ be
the dual basis, determined by $e_k^*(e_\ell)=\delta_{k,\ell}$.
There exists a vector space automorphism
$\alpha\colon \R^n\to\R^n$ such that
\[
\alpha^*(\lambda_j)=e_j^*\quad\mbox{for all $\,j\in \{1,\ldots,i\}$,}
\]
using the dual linear map $\alpha^*\colon (\R^n)^*\to(\R^n)^*$,
$\lambda\mto \lambda\circ\alpha$.
Note that
\[
\{w\in \R^n\colon (\forall j\in \{1,\ldots,i\})\;e_j^*(w)\geq 0\}
=[0,\infty[^i\times\R^{n-i}.
\]
Now $\alpha^{-1}(W)$ is open in $\R^n$
and $z\in K_1\cap\cdots\cap K_i$ for $z\in \R^n$
if and only if
\[
0\leq \lambda_j(z)=\lambda_j(\alpha(\alpha^{-1}(z))))=e_j^*(\alpha^{-1}(z))
\]
for all $j\in \{1,\ldots,i\}$, i.e., if and only if
$\alpha^{-1}(z)\in [0,\infty[^i\times\R^{n-i}$.
Thus $\alpha^{-1}$ maps $K_1\cap\cdots\cap K_i$ onto
$[0,\infty[^i\times \R^{n-i}$. As a consequence,
$\alpha^{-1}$ restricts to a $C^\infty$-diffeomorphism
$\kappa$ from the relatively open subset
\[
W\cap M=W\cap K_1\cap\cdots\cap K_i
\]
of~$M$ onto the relatively open subset $\alpha^{-1}(W)\cap ([0,\infty[^i\times\R^{n-i})$
of $[0,\infty[^i\times\R^{n-i}$.\\[2.3mm]
``(b)$\Rightarrow$(a)'':
This implication is immediate from the next lemma,
which provides additional information.
\end{proof}
\begin{la}\label{super-la}
Let $M$ be a cube-like $n$-polytope and
$\kappa\colon U\to [0,1[^i\times \,]{-1},1[^{n-i}$
be a standard chart around $x\in M$.
If $i:=\ind_M(x)>0$
and $F_1,\ldots, F_m$ are the facets
of~$M$ containing~$x$,
then $m=i$ holds and
there is a permutation $\pi$ of $\{1,\ldots, i\}$
such that
\begin{equation}\label{face-subm}
\kappa^{-1}(F_{i,k}\times \,]{-1},1[^{n-i})=U\cap F_{\pi(k)}
\end{equation}
for all $k\in \{1,\ldots, i\}$,
writing $F_{i,k}:=\{(y_1,\ldots, y_i)\in [0,1[^i\colon y_k=0\}$.
For each facet $F$ of~$M$ such that $F\not\in\{F_1,\ldots, F_i\}$,
we have $F\cap U=\emptyset$.
\end{la}
\begin{proof}
The face $M(x)$ of $M$ generated by~$x$
has dimension $n-i$.
Let $F_1,\ldots, F_m$
be the mutally distinct facets
of~$M$ which contain~$x$.
A facet $F$ of~$M$ contains
$x$ if and only if $M(x)\sub F$.
Hence $m\geq n-(n-i)=i$ by
\cite[Theorem~10.4]{Bro83}.
Let $S=\bigcup_{k=1}^m\algint(F_k)$.
For each $j\in \{1,\ldots, m\}$, the connected set
\[
\algint(F_j)=S\setminus\bigcup_{k\not=j}F_k
\]
is relatively open in~$S$
(where we used that $\algint(F_j)\cap F_k=\emptyset$
for all $k\in\{1,\ldots, m\}$ such that $k\not=j$).
Hence $\algint(F_1),\ldots,\algint(F_m)$
are the connected components of~$S$.
Let
$F_{i,k}^*$ be the set of all $(y_1,\ldots,y_i)\in F_{i,k}$
such that $y_j>0$ for all $j\in \{1,\ldots, i\}\setminus\{k\}$.
Then
\[
\kappa^{-1}(F_{i,k}^*\times\,]{-1},1[^{n-i})\sub S
\]
and thus $\kappa^{-1}(F_{i,k}^*\times\,]{-1},1[^{n-i})\sub \algint(F_{\pi(k)})$
for some $\pi(k)\in \{1,\ldots,i\}$,
the preimage being connected. As $F_{\pi(k)}$ is closed
and $F_{i,k}^*$ is dense in $F_{i,k}$, we infer
\begin{equation}\label{half}
\kappa^{-1}(F_{i,k}\times\,]{-1},1[^{n-i})\sub F_{\pi(k)}.
\end{equation}
Since $U\cap\algint(F_j)$ is dense in the $x$-neighbourhood
$U\cap F_j$ in~$F_j$, we must have $U\cap\algint(F_j)\not=\emptyset$
for all $j\in \{1,\ldots, m\}$. Hence $\pi\colon \{1,\cdots, i\}\to\{1,\ldots,m\}$
is surjective and thus $m\leq i$,
whence $m=i$.
Being a surjective self-map
of a finite set, $\pi$ is a bijection.\\[2.3mm]
If $z\in U$ and $\ind_M(z)=1$,
then $\kappa(z)$ has index~$1$ in $\kappa(U)$,
whence
\[
\kappa(z)\in F_{i,k}^*\times \,]{-1},1[^{n-i}\quad\mbox{for some $k\in \{1,\ldots,i\}$}
\]
(as the union of these sets equals $\{y\in \kappa(U)\colon \ind_{\kappa(U)}(y)=1\}$)
and thus $z\in U\cap \algint(F_{\pi(k)})$.
Hence
\[
\kappa(U\cap\algint(F_{\pi(k)}))=F_{i,k}^*\times \,]{-1},1[^{n-i}.
\]
As $\algint(F_{\pi(k)})$ is dense in $F_{\pi(k)}$,
the map $\kappa$ is continuous and $F_{i,k}\times \,]{-1},1[^{n-i}$
is closed in $\kappa(U)$, we deduce that $\kappa(U\cap F_{\pi(k)})\sub F_{i,k}\times\,
]{-1},1[^{n-i}$. Since
$\kappa^{-1}(F_{i,k}\times \,]{-1},1[^{n-i})\sub U\cap F_{\pi(k)}$ by (\ref{half}),
the equality in~(\ref{face-subm}) follows.
\end{proof}
\begin{rem}\label{rem:mf_faces}
Since cube-like and simple embedded polytopes coincide by \Cref{main-equivalence}, we obtain an even stronger statement: A cube-like $n$-polytope in $\R^n$ is even a so called manifold with faces, cf.\ \cite[Section 1.1]{Jae68}. This subclass requires of the manifold (now translated to the specific case of an embedded polytope and its boundary) that the $k$-faces lie in the intersection of $n-k$ facets, i.e. that they satisfy \Cref{defn-sim} (c). However, we will not directly need this additional information in the following sections.
\end{rem}
\section{Proof of Theorem~\ref{thm-B} and related results}
\begin{numba}\label{extend-zero}
If $E$ is a finite-dimensional vector
space, $U\sub E$ a locally convex subset with dense
interior, $K\sub U$ a closed subset,
$m\in \N_0\cup\{\infty\}$ and $Y$
a locally convex space,
then
\[
C^m_K(U,Y):=\{f\in C^m(U,Y)\colon \Supp(f)\sub K\}
\]
is a closed vector subspace of $C^m(U,Y)$.
If $V\sub E$ is a locally convex subset with
dense interior such that $K\sub V$ and $V\sub U$,
then the restriction map
\[
\rho_{V,U}\colon C^m_K(U,Y)\to C^m_K(V,Y),\quad f\mto f|_V
\]
is continuous and linear (see \ref{basicpull})
and a homeomorphism as the map
$C^m_K(V,Y)\to C^m_K(U,Y)$ extending functions by $0$
(which is the inverse of $\rho_{V,U}$)
is continuous. This is clear for the topology of compact
convergence if $m=0$. Since $d^kf\in C_{K\times E^k}(U\times E^k,F)$
for each $f\in C^m_K(U,F)$ and $k\in \N$ such that $k\leq m$
(and likewise for $V$ in place of~$U$),
the assertion follows. Cf.\ also \cite[Lemma~4.1.9]{GN25}.
\end{numba}
\begin{numba}\label{add-on}
In the situation of Theorem~\ref{thm-B},
if $N$ is a face of~$M$ of dimension $>\ell$,
write $\cF(N,\ell)$ for the set
of all $\ell$-faces of~$N$.
We shall see that $\sigma$ in Theorem~B
can be chosen with the following additional property:
For each $f=(f_F)_{f\in\cF}\in \cE$ and face $N$ of $M$ of dimension $>\ell$,
\[
\sigma(f)(N)\;\sub \;\; \Spann \!\!\!\! \bigcup_{F\in \cF(N,\ell)}\!\! f_F(F).
\]
\end{numba}
\noindent
{\bf Proof of Theorem~\ref{thm-B}.}
We show by induction on $j\in \N$ that
the assertion of the theorem holds and
that the continuous linear right inverse
can be chosen as in~\ref{add-on}
for each integer $n\geq j+1$,
with $\ell:=n-j$.\\[2.3mm]
The case $j=1$. Then $\ell=n-1$
and $\cF$ is the set of facets of~$M$,
for each finite-dimensional vector space~$E$
and $n$-polytope $M\sub E$.
For each $z\in M$, we let $i(z):=\ind_M(z)$
and pick a standard chart
\[
\kappa_z\colon U_z\to [0,1[^{i(z)}\times\,]{-1},1[^{n-i(z)}
\]
around~$z$, as in Lemma~\ref{sta-cha}.
If $i(z)>0$,
let $F^z_1,\ldots, F^z_{i(z)}$
be the facets of~$M$ containing~$z$;
after a permutation of the indices,
we may assume that
\[
\kappa_z^{-1}(F_{i(z),k}\times\,]{-1},1[^{n-i(z)})=U_z\cap F_k^z
\]
for $k\in \{1,\ldots, i(z)\}$,
see Lemma~\ref{super-la}.
For each $f=(f_F)_{F\in \cF}\in \cE$,
we have
\begin{equation}\label{toolong}
\Xi_z(f):=(f_{F^z_k}\circ \kappa_z^{-1}|_{F_{i(z),k}\times ]{-1},1[^{n-i(z)}})_{k=1}^{i(z)}
\in E_{n,i(z)}
\end{equation}
with notation as in \ref{reunot}
($i(z)$ playing the role of~$i$).
In fact, if $k,k'\in \{1,\ldots, i(z)\}$,
$x\in F_{i(z),k}\cap F_{i(z),k'}$,
and $q\in\;]{-1},1[^{n-i(z)}$,
then
\[
\kappa_z^{-1}(x,q)\in F^z_k\cap F^z_{k'},
\]
whence $f_{F^z_k}(\kappa_z^{-1}(x,q))
=f_{F^z_{k'}}(\kappa_z^{-1}(x,q))$
as $(f_F)_{F\in \cF}\in \cE$.
Thus~(\ref{toolong}) holds.
As a consequence of \ref{basicpull},
the map
\[
\Xi_z\colon \cE\to E_{n,i(z)}
\]
is continuous and linear and so is
\[
(\kappa_z)^*\circ \Phi_{n,i(z)}\circ \Xi_z\colon \cE\to C^m([0,1[^{i(z)}\times\;]{-1},1[^{n-i(z)},Y),
\]
using $\Phi_{n,i(z)}$ as in~\ref{reunot}
and the map
\[
(\kappa_z)^*\colon C^m([0,1[^{i(z)}\times \;]{-1},1[^{n-i(z)},Y)\to C^m(U_z,Y),\;\;
g\mto g\circ \kappa_z
\]
which is continuous linear (cf.\ \ref{basicpull}).
Note that, for each $k\in \{1,\ldots, i(z)\}$
and $x\in U_z\cap F^z_k$, we have $\kappa_z(x)\in F_{i(z),k}\times\;]{-1},1[^{n-i(z)}$
and thus
\[
((\kappa_z)^*\circ \Phi_{n,i(z)}\circ\Xi_z)(f)(x)
=\Phi_{n,i(z)}(\Xi_z(f))(\kappa_z(x))=
f_{F^z_k}(\kappa_z^{-1}(\kappa_z(x))=f_{F^z_k}(x),
\]
that is,
\begin{equation}\label{halfway}
((\kappa_z)^*\circ \Phi_{n,i(z)}\circ\Xi_z)(f)(x)
=f_{F^z_k}(x).
\end{equation}
There exists a smooth partition of unity $(h_z)_{z\in M}$
on $M$ with $S_z:=\Supp(h_z)$ $\sub U_z$ for all $z\in M$.
Then
\[
Z:=\{z\in M\colon \mbox{$\ind_M(z)>0$ and $h_z\not=0$}\}
\]
is a finite subset of~$M$.
The multiplication operator
\[
\mu_z\colon C^m(U_z,F)\to C^m_{S_z}(U_z,F),\quad g \mto h_z\cdot g
\]
is continuous and linear (see \cite[Lemma~4.1.39]{GN25});
also the operator
\[
\ve_z\colon C^m_{S_z}(U_z,Y)\to C^m_{S_z}(M,Y)\sub C^m(M,Y)
\]
which extends functions by~$0$ is continuous linear
(cf.\ \ref{extend-zero}).
Hence
\[
\alpha_z:=\ve_z\circ \mu_z\circ (\kappa_z)^*\circ
\Phi_{n,i(z)}\circ\Xi_z\colon \cE\to C^m(M,Y)
\]
is a continuous linear map. As a consequence, also the map
\[
\sigma:=\sum_{z\in Z}\alpha_z\colon \cE\to C^m(M,Y)
\]
is continuous and linear.
Let $f=(f_F)_{F\in \cF}\in\cE$.
Let $F$ is a facet of~$M$ and $x\in F$.
Given $z\in M$, $h_z(x)>0$ implies that $x\in U_z$,
whence $\ind_M(z)>0$ (as $U_z\sub\algint(M)$
if $\ind_M(z)=0$). Thus $z\in Z$, entailing that
\begin{equation}\label{sumone}
\sum_{z\in Z_x}h_z(x)=1
\end{equation}
with $Z_x:=\{z\in Z\colon h_z(x)>0\}$.
For each $z\in Z_x$,
we have $F\cap U_z\not=\emptyset$
as $x\in F\cap U_z$ and
hence $F=F^z_{k(z)}$ for some
$k(z)\in \{1,\ldots,i(z)\}$,
by the final condition in Lemma~\ref{sta-cha}.
Using (\ref{halfway}) and (\ref{sumone}), we deduce that
\[
\sigma(f)(x)=\sum_{z\in Z_x}h_z(x)
((\kappa_z)^*\circ
\Phi_{n,i(z)}\circ\Xi_z)(f)(x)=
\sum_{z\in Z_x}h_z(x)f_F(x)=f_F(x).
\]
Thus $(r\circ \sigma)(f)=f$.
Let $x\in M$. If $\ind_M(x)>0$,
then $x\in G$ for some facet $G$ of~$M$
and $\sigma(f)(x)=f_G(x)$, whence
\begin{equation}\label{element-span}
\sigma(f)(x)\;\in\;\Spann\;\;\bigcup_{F\in \cF}f_F(F)=:W.
\end{equation}
For each $z\in M\setminus\algint(M)$,
each component of $\Xi_z(f)$
is a function with values in~$W$.
Hence also $\Phi_{n,i}(\Xi_z(f))$
is a function with values in~$W$,
by Remark~\ref{rem-in-span}.
If $x\in \algint(M)$, then $\sigma(f)(x)$
is a linear compination
of function values
of functions of the form
$\Phi_{n,i}(\Xi_z(f))$,
whence (\ref{element-span}) also
holds in this case.
Thus $\sigma$ satisfies the condition
formulated in \ref{add-on}.\\[2.3mm]
Continuing with the induction step, let us write $\cF(M,\ell)$, $\cE(M,\ell)$,
and $r_{M,\ell}$
in place of $\cF$, $\cE$,
and $r$, respectively,
in the situation of Proposition~B. Let us write
$\sigma_{M,\ell}$ in place of~$\sigma$,
if it exists. Let $j\geq 2$ be an integer
such that the assertion holds
for $j-1$ in place of~$j$.
Let $n\geq j+1$, $\ell:=n-j$, and $M$ be an $n$-polytope
in a finite-dimensional vector space~$E$.
For each $N\in \cF(M,\ell+1)$,
its set $\cF(N,\ell)$ of $\ell$-dimensional
faces is a subset of $\cF(M,\ell)$.
As $\cF(N,\ell)$ is the set
of facets of~$N$, the base of the induction
furnishes a continuous linear right inverse
\[
\sigma_{N,\ell}\colon \cE(N,\ell)\to C^m(N,Y)
\]
for $r_{N,\ell}$.
By the case $j-1$, we have a continuous linear
right inverse
\[
\sigma_{M,\ell+1}\colon \cE(M,\ell+1)\to C^m(M,Y)
\]
for $r_{M,\ell+1}$.
The map
\[
\prod_{F\in \cF(M,\ell)}C^m(F,Y)\to\prod_{F\in \cF(N,\ell)}C^m(F,Y),\quad
(f_F)_{F\in \cF(M,\ell)}\mto(f_F)_{F\in \cF(N,\ell)}
\]
is continuous linear and restricts to a map
\[
R_{N,M}\colon \cE(M,\ell)\to\cE(N,\ell).
\]
For each $f=(f_G)_{G\in \cF(M,\ell)}\in \cE(M,\ell)$,
we have
\begin{equation}\label{pre-conditio}
(\sigma_{N,\ell}\circ R_{N,M})(f)|_F=\sigma_{N,\ell}((f_G)_{G\in \cF(N,\ell)})|_F=
f_F
\end{equation}
for each $f\in \cF(N,\ell)$.
The map
\[
\alpha:=(\sigma_{N,\ell}\circ R_{N,M})_{N\in\cF(M,\ell+1)}
\colon \cE(M,\ell)\to\prod_{N\in \cF(M,\ell+1)}C^m(N,Y)
\]
is continuous linear. We claim that $\alpha$ has image
in $\cE(M,\ell+1)$; we can therefore
consider its corestriction
\[
\beta\colon \cE(M,\ell)\to \cE(M,\ell+1),\quad f\mto \alpha(f).
\]
Then $\beta$ and
\[
\sigma_{M,\ell}:=\sigma_{M,\ell+1}\circ \beta\colon \cE(M,\ell)\to C^m(M,Y)
\]
are continuous linear mappings.
Let $f=(f_F)_{F\in \cF(M,\ell)}\in \cE(M,\ell)$
and $g:=\beta(f)$; write $g=(g_N)_{N\in \cF(M,\ell+1)}$.
For each $F\in \cF(M,\ell)$,
there exists $N\in \cF(M,\ell+1)$
such that $F\sub N$ (see \cite[Corollary~9.7]{Bro83}).
Then
\[
\sigma_{M,\ell}(f)|_F=(\sigma_{M,\ell}(f)|_N)|_F
=(\sigma_{M,\ell+1}(g)|_N)|_F=g_N|_F=f_F.
\]
Thus $(r_{M,\ell}\circ \sigma_{M,\ell})(f)=f$,
as $F\in \cF(M,\ell)$ was arbitrary.
To see that $\sigma_{M,\ell}$ (in place of~$\sigma$)
satisfies the condition of~\ref{add-on},
let $f=(f_F)_{F\in \cF(M,\ell)}\in \cE(M,\ell)$.
Let $g_N:=\sigma_{N,\ell}(R_{N,M}(f))$
for $N\in \cF(M,\ell+1)$; thus
$g:=(g_N)_{N\in\cF(M,\ell+1)}=\beta(f)$.
By the base of the induction,
\[
g_N(N)\, \sub  \;\; \Spann\!\!\! \bigcup_{F\in \cF(N,\ell)}f_F(F).
\]
Since $\sigma_{M,\ell}(f)|_N=\sigma_{M,\ell+1}(g)|_N=g_N$,
we see that the condition of \ref{add-on}
is satisfied for each $N\in \cF(M,\ell+1)$.
If $K$ is a face of~$M$ of dimension $d>\ell+1$,
then
\[
\sigma_{M,\ell}(f)(K)\, =\, \sigma_{M,\ell+1}(g)(K)\sub\;\; \Spann\!\!\!
\bigcup_{N\in\cF(K,\ell+1}g_N(N)
\]
by induction, where
\[
g_N(N)\, \sub\;\; \Spann\!\!\! \bigcup_{F\in \cF(N,\ell)}f_F(F)\,
\sub\;\; \Spann\!\!\!
\bigcup_{F\in \cF(K,\ell)}f_F(F)=:W.
\]
Hence $\sigma_{M,\ell}(f)(K)\sub W$, establishing the condition of \ref{add-on}
for~$K$ in place of~$N$.\\[2.3mm]
It remains to prove the claim.
Let $f=(f_F)_{F\in \cF(M,\ell)}\in \cE(M,\ell)$
and $g:=\alpha(f)$; write $g=(g_N)_{N\in\cF(M,\ell+1)}$.
For all $N_1,N_2\in \cF(M,\ell+1)$,
we show that
\begin{equation}\label{conditio}
g_{N_1}|_{N_1\cap N_2}=g_{N_2}|_{N_1\cap N_2}.
\end{equation}
Excluding trivial cases,
we may assume that $N_1\not=N_2$ and $N_1\cap N_2\not=\emptyset$.
Then
$N_1\cap N_2$ is a face of $M$
of dimension $\leq \ell$.
By \cite[Corollary~9.7]{Bro83},
there exists a
face $F$ of $M$ of dimension~$\ell$
such that $N_1\cap N_2\sub F$.
Then
\[
g_{N_1}|_F=f_F=g_{N_2}|_F,
\]
by (\ref{pre-conditio}),
from which (\ref{conditio}) follows. $\,\square$\\[2.3mm]
\begin{numba}\label{facesandstratvec}
If $M\not=\emptyset$ is a polytope in a finite-dimensional
real vector space~$E$,
then the vector subspace
\[
E_M:=\aff(M)-x
\]
of~$E$
is independent of $x\in M$.
Given a non-empty face~$F$ of~$M$,
we define $E_F\sub E$ in the same way. For $x\in M$,
we write $M(x)$ for the face of~$M$
generated by~$x$ (the smallest face containing~$x$)
and abbreviate
\[
E_x:=E_{M(x)}.
\]
We call
\[
C^\infty_{\str}(M,E):=\{f\in C^\infty(M,E)\colon (\forall x\in M)\; f(x)\in E_x\}
\]
the space of \emph{stratified vector fields};
we endow it with the topology induced
by the compact-open $C^\infty$-topology on $C^\infty(M,E)$.
It is unchanged if we replace~$E$
with a vector subspace of~$E$ which contains $\aff(M)$.
\end{numba}
\begin{numba}\label{span-inside}
If $F$ and $N$ are non-empty faces of $M\sub E$
such that $F\sub N$, then $\aff(F)\sub \aff(N)$
and hence $\aff(F)-x\sub \aff(N)-x$ for $x\in F$, whence
\[
E_F\sub E_N.
\]
As a consequence,
\[
\Spann \bigcup_{F\in \cF}E_F\sub E_N
\]
for each set $\cF$ of non-empty faces of~$N$.
\end{numba}
\begin{la}\label{criterion-stratified}
Let $n\geq 2$ be an integer,
$M$ be an $n$-polytope in a finite-dimensional vector space~$E$
and $\ell\in \{1,\ldots, n-1\}$.
Let $f\colon M\to E$ be a smooth function; assume that
\begin{itemize}
\item[\rm(a)]
$f|_F\in C^\infty_{\str}(F,E)$ for each $\ell$-face
$F$ of~$M$; and
\item[\rm(b)]
For each face $N$ of~$M$
of dimension $>\ell$,
\[
f(N)\, \sub\;\, \Spann\!\!\! \bigcup_{F\in \cF(N,\ell)}f(F),
\]
where $\cF(N,\ell)$ is the set of $\ell$-faces of~$N$.
\end{itemize}
Then $f\in C^\infty_{\str}(M,E)$.
\end{la}
\begin{proof}
Let $x\in M$ and $i:=\ind_M(x)$.
Thus $M(x)$ has dimension $n-i$.
If $n-i\leq \ell$,
then $M(x)\sub F$
for an $\ell$-face $F$ of~$M$
by \cite[Corollary 9.7]{Bro83}.
Since $f|_F\in C^\infty_{\str}(F,E)$ by~(a),
we deduce that $f(x)=f|_F(x)\in E_x$.
For each face $F$ of $M$
of dimension $\ell$, we have $M(y)\sub F$
for each $y\in F$ and thus $E_y=E_{M(y)}\sub F$.
Moreover, $n-\ell\leq \ind_M(y)$. Hence
\begin{equation}\label{image-face}
f(F)=\bigcup_{y\in F}f(y)\sub \bigcup_{y\in F}E_y\sub F.
\end{equation}
If $n-i>\ell$, then
\[
f(x)\in f(M(x))\, \sub\;\, \Spann\!\!\! \bigcup_{F\in \cF(M(x),\ell)}\!\! f(F)
\, \sub\;\, \Spann\!\!\! \bigcup_{F\in \cF(M(x),\ell)}\!\!
E_F
\;\sub E_{M(x)}=E_x,
\]
using (b),
(\ref{image-face}),
and \ref{span-inside}. Thus~$f$ is a stratified vector field.
\end{proof}
\begin{cor}\label{vecfield-cor}
Let $E$ be a finite-dimensional
vector space, $n$ a positive integer,
$M\sub E$ be a simple $n$-polytope,
$\ell\in \{1,\ldots, n-1\}$
and $\cF$ be the set of $\ell$-faces of~$M$.
Let $\cV$ be the closed vector subspace
of
$\prod_{F\in \cF}C^\infty_{\str}(F,E)$
consisting of all $(f_F)_{F\in \cF}$ such that
$f_F|_{F\cap G}=f_G|_{F\cap G}$ for all $F,G\in \cF$.
Then the map
\[
R \colon C^\infty_{\str}(M,E)\to \cV,\quad f\mto (f|_F)_{F\in \cF}
\]
is continuous linear and has a continuous linear
right inverse $\tau\colon \cV\to C^\infty_{\str}(M,E)$.
\end{cor}
\begin{proof}
As a consequence of~\ref{basicpull},
$R$ is continuous and linear.
For $Y:=E$ and $m:=\infty$,
let $\cE\sub\prod_{F\in \cF}C^\infty(M,E)$
and $r\colon C^\infty(M,E)\to\cE$ be as in Theorem~B.
The theorem furnishes a continuous linear
right inverse $\sigma\colon \cE\to C^\infty(M,E)$ for~$r$;
we may assume that $\sigma$ satisfies the condition
described in~\ref{add-on}.
Note that $C^\infty_{\str}(F,E)\sub C^\infty(F,E)$
for each $F\in \cF$ and $\cV\sub \cE$.
We claim that
\[
\sigma(f)\in C^\infty_{\str}(M,E)
\]
for each $(f=f_F)_{F\in \cF}\in\cV$.
If this is true, then the co-restriction
$\tau\colon \cV\to C^\infty_{\str}(M,E)$,
$f\mto \sigma(f)$ is a continuous linear right inverse
for~$R$.
To prove the claim,
let $f=(f_F)_{F\in \cF}\in \cV$ and abbreviate $g:=\sigma(f)$.
For each $F\in \cF$, we have
\[
g|_F=f_F\in C^\infty_{\str}(F,E),
\]
whence $g$
satisfies condition (a)
of Lemma~\ref{criterion-stratified}
(with $g$ in place of~$f$). Moreover, $g(F)=f_F(F)\sub E_F$.
For each face $N$ of $M$ of dimension
$>\ell$, we have
\[
g(N)=\sigma(f)(N)\sub\Spann \bigcup_{F\in \cF(N,\ell)}f_F(F)
=\Spann\bigcup_{F\in \cF(N,\ell)}g(F),
\]
using \ref{add-on} for the inclusion.
Hence also condition~(b)
of Lemma~\ref{criterion-stratified}
is satisfied and thus $g\in C^\infty_{\str}(M,E)$.
\end{proof}
\noindent
Let $M$ be an $n$-polytope.
If $M\sub \R^n$, we can use $\id_M$
as a global chart to consider
$M$ as an $n$-dimensional
locally polyhedral manifold
in the sense of \cite{Glo23},
or an $n$-dimensional smooth manifold with rough boundary
in the sense of \cite{GN25}.
The latter are defined like manifolds
with corners, except that the $V_\phi$ in Definition~\ref{mfd-corners}
need to be replaced with locally convex subsets of~$\R^n$
with dense interior.
If~$M$ is an arbitary $n$-polytope, we can use
an affine diffeomorphism $f\colon \R^n\to\aff(M)$
to transport the manifold structure from $f^{-1}(M)$
to~$M$.
\begin{rem}\label{converse}
If $M$ is an $n$-polytope
which is not simple
(whence $n\geq 3$),
then the image
of $\rho$ is a proper subset of $\cE$
in Theorem~B (whence its conclusion becomes false)
whenever $Y\not=\{0\}$ and $\ell=1$.
Likewise, the image of $R$ is a proper
subset of $\cV$ in Corollary~\ref{vecfield-cor}
(whence its conclusion becomes false)
for $\ell=1$.
\end{rem}
\noindent
To see this, we may assume that $M$
is an $n$-polytope in $\R^n$.
We let $x_0$ be a vertex of~$M$
such that the number $m$ of edges
containing~$x_0$ exceeds~$n$.
Let $\cF$ be the set of all edges of~$M$
and $F_1,\ldots, F_m$
be the edges containing~$x_0$.
Let $x_j$ the other vertex of~$F_j$
for $J\in \{1,\ldots,m\}$.
After a permutation of the indices,
we may assume that
\[
x_m-x_0\in \Spann\{x_1-x_0,\ldots, x_{m-1}-x_0\},
\]
say $x_m-x_0=\sum_{j=1}^{m-1}\lambda_j(x_j-x_0)$
with $\lambda_1,\ldots, \lambda_{m-1}\in \R$.
For each $F\in \cF\setminus \{F_m\}$,
we let $f_F\in C^\infty_{\str}(F,\R^n)$
be the function $f_F=0$.
We let $h\colon f_{F_m}\colon F_m\to \R$
be a smooth function such that $h(x_0)=h(x_1)=0$
and $h(x_0+t(x-x_0))=t$ for small $t\geq 0$.
Concerning Theorem~\ref{thm-B},
we choose $v\in Y\setminus\{0\}$.
Concerning Corollary~\ref{vecfield-cor},
we choose $v:=x_m-x_0$.
Then $f_{F_m}(t):=h(t)v$ defines a function
$f_{F_m}\in C^\infty(F_m,Y)$,
respectively, a function
$f_{F_m}\in C^\infty_{\str}(F_m,\R^n)$.
Moreover, we have $f:=(f_F)_{F\in \cF}\in \cE$,
respectively, $f:=(f_F)_{F\in \cF}\in \cV$.
If we had $f=r(g)$ for
some $g\in C^\infty(M,Y)$
(or $g=R(g)$ for some $C^\infty_{\str}(M,\R^n)$),
then $g|_{F_j}=f_{F_j}=0$ for $j\in \{1,\ldots,m-1\}$,
whence
\[
dg(x_0,x_j-x_0)=0.
\]
Hence $dg(x_0,x_m-x_0)=\sum_{j=1}^{m-1}\lambda_j\, dg(x_0,x_j-x_0)=0$.
But
\[
\frac{d}{dt}\Big|_{t=0}g(x_0+t(x_m-x_0))=\frac{d}{dt}\Big|_{t=0}(tv)=v\not=0,
\]
contraction. Thus $g$ cannot exist.
\section{Proof of Theorem~\ref{thm-A} and Corollary~\ref{new-cor}}
We shall use a simple fact
(see, e.g., \cite[Corollary 1.7.13]{GN25}).
\begin{numba}\label{basicpush}
Let $U$ be a locally convex
subset with dense interior in a locally convex space~$E$.
Let $\alpha\colon F_1\to F_2$ be a continuous
linear map between locally convex spaces.
Then the following map is continuous and linear:
\[
\alpha_*\colon C^\infty(U,F_1)\to C^\infty(U,F_2),\quad
f\mto \alpha\circ f.
\]
\end{numba}
\begin{numba}
Let $n$ be a positive integer
and $M$ be an $n$-polytope
in $\R^n$.
Then
\[
\Omega_M:=\{\phi-\id_M\colon \phi\in \Diff^{\fr}(M)\}
\]
is an open subset of $C^\infty_{\str}(M,\R^n)$
and the map
\[
\theta_M\colon \Diff^{\fr}(M)\to\Omega_M,\quad \phi\mto\phi-\id_M
\]
is a bijection
which can be used as a global chart
for a smooth manifold structure on
$\Diff^{\fr}(M)$
making it a Lie group~\cite{Glo23};
the modeling space is $C^\infty_{\str}(M,\R^n)$.
There is a unique Lie group
structure on $\Diff(M)$
making $\Diff^{\fr}(M)$
an open submanifold (see \cite{Glo23}).
\end{numba}
\begin{numba}\label{DiffP}
If $E$ is a finite-dimensional vector
space and $M\sub E$ an $n$-polytope,
we choose an affine bijection
\[
A\colon \R^n\to \aff(M).
\]
Then $P:=A^{-1}(M)$ is an
$n$-polytope in $\R^n$;
we define $\Omega_P\sub C^\infty_{\str}(P,\R^n)$
and $\theta_P\colon \Diff^{\fr}(P)\to\Omega_P$
as before.
There are $b\in E$
and a linear map $\alpha\colon \R^n\to E$
such that
\[
A(x)=\alpha(x)+b\quad\mbox{for all $\,x\in \R^n$.}
\]
Then $A(\aff(F))=\aff(A(F))$ for each non-empty face
$F$ of $P$ and
\[
\alpha((\R^n)_F)=E_{A(F)}
\]
holds for the corresponding vector subspaces.
By the preceding, the map
\[
\beta\colon C^\infty_{\str}(M,E)\to C^\infty_{\str}(P,\R^n),\quad f\mto \alpha^{-1}\circ
f\circ A|_P
\]
is a bijection; it is an isomorphism
of topological vector spaces
as a consequence of \ref{basicpull}
and \ref{basicpush}.
The map
\[
C_A\colon \Diff(P)\to\Diff(M),\;\;
\phi\mto A\circ \phi\circ A^{-1}
\]
is an isomorphism of groups
which takes $\Diff^{\fr}(P)$ onto
$\Diff^{\fr}(M)$.
Let $c_A$ be its restriction
to an isomorphism $\Diff^{\fr}(P)\to\Diff^{\fr}(M)$.
If we define
\[
\Omega_M:=\{\phi-\id_M\colon \phi\in \Diff^{\fr}(M)\},
\]
then $\Omega_M$ is a subset of $C^\infty_{\str}(M,E)$
and the map
\[
\theta_M\colon \Diff^{\fr}(M)\to \Omega_M,\quad \phi\mto \phi-\id_M\}
\]
is a bijection. Then
\begin{equation}\label{commute}
\theta_M\circ c_A=   \beta\circ \theta_P,
\end{equation}
entailing that $\Omega_M$ is open in $C^\infty_{\str}(M,E)$.
If we give $\Diff^{\fr}(M)$ the smooth manifold
structure modeled on $C^\infty_{\str}(M,E)$
making $\theta_M$
a $C^\infty$-diffeomorphism, we deduce from~(\ref{commute})
that $c_A$ is a $C^\infty$-diffeomorphism,
whence the smooth manifold structure makes
$\Diff^{\fr}(M)$ a Lie group.
Using $C_A$, we can transport the Lie group
structure from $\Diff(P)$ to $\Diff(M)$
and obtain a Lie group structure on $\Diff(M)$
with $\Diff^{\fr}(M)$ as an open submanifold.
\end{numba}
\noindent
{\bf Proof of Theorem~\ref{thm-A}.}
Let $E$ be a finite-dimensional
vector space and an $n$-polytope $M\sub E$
as well as $\ell$, $\cF$, and $\rho$
be as in Theorem~A. Let $\Omega_M$
and $\theta_M$ be as in~\ref{DiffP}.
Let $\cV$, $R$, and its continuous linear right inverse~$\tau$
be as in Corollary~\ref{vecfield-cor}.
As the continuous linear map $R\colon C^\infty_{\str}(M,E)\to \cV$
has a continuous linear right inverse,
it is an open map. Hence $R(\Omega_M)$ is open
in $\cV$, showing that $R(\Omega_M)$ is a submanifold
of~$\cV$ and hence of $\prod_{F\in\cF}C^\infty_{\str}(F,E)$.
Being a continuous linear map with a continuous linear
right inverse, $R$ is a submersion
(in the sense of \cite[Definition~4.4.8]{Ham82}), and
hence also its restriction
$R|_{\Omega_M}\colon \Omega_M\to R(\Omega_M)$
to the open subset~$\Omega_M$ is a submersion.
Since
\[
\left(\prod_{F\in \cF}\theta_F\right)\circ \rho=R \circ\theta_M,
\]
the $C^\infty$-diffeomorphism $\theta:=\prod_{F\in \cF}\theta_F$
takes $\im(\rho)$ onto $R(\Omega_M)$.
Notably, $R(\Omega_M)$ is contained in the open subset
$\prod_{F\in \cF}\Omega_F$ of $\prod_{F\in\cF}C^\infty_{\str}(F,E)$,
whence $R(\Omega_M)$ can be regarded as a submanifold of
$\prod_{F\in\cF}\Omega_F$.
As $\theta$ is a $C^\infty$-diffeomor\-phism,
we infer that $\im(\rho)$ is a submanifold of
$\prod_{F\in \cF}\Diff^{\fr}(F)$
and that the submanifold structure
makes $\Theta:=\theta|_{\im(\rho)}\colon \im(\rho)\to R(\Omega_M)$
a $C^\infty$-diffeomorphism. Since $\Theta\circ \rho|^{\im(\rho)}=R\circ \theta_M$
is a submersion, also $\rho|^{\im(\rho)}\colon \Diff^{\fr}(M)\to \im(\rho)$
is a submersion. Since $\prod_{F\in \cF}\Diff^{\fr}(F)$
is a Lie group, its subgroup and submanifold $\im(\rho)$
also is a Lie group.
Recall that the connected component
of $\id_M$ in $\Diff^{\fr}(M)$
equals the connected component
$\Diff(M)_0$ of $\Diff(M)$ (cf.\ \cite{Glo23}).
If $\ell=1$,
then $\cV=\prod_{F\in \cF}C^\infty_{\str}(F,E)$,
whence $R(\Omega_M)$ is open in $\prod_{F\in \cF}C^\infty_{\str}(F,E)$.
As a consequence, $\im(\rho)=\theta^{-1}(R(\Omega_M))$
is open in $\prod_{F\in \cF}\Diff^{\fr}(F)=G$. Thus $\rho$
is open as a map to $G$.
Hence $\rho(\Diff(M)_0)$ is an open
subgroup of~$G$ and hence contains the identity component
$G_0$ of~$G$.
As $\rho(\Diff(M)_0)$ is connected,
it is contained in~$G_0$. Thus
$\rho(\Diff(M)_0)=G_0=\prod_{F\in \cF}\Diff(F)_0$. $\,\square$\\[2.3mm]
The following lemma can be proved using standard
arguments.
\begin{la}\label{from-open}
Let $G$ be a Lie group, $U$ be an open subgroup
of~$G$ and $N$ be a closed normal subgroup of~$G$.
Consider $U/(U\cap N)$ as a subset of $G/N$,
identifying $g(U\cap N)$ with $gN$ for $g\in U$.
Let $q\colon G\to G/N$, $g\mto gN$ be the canonical
map. If there exists a smooth manifold
structure on $U/U\cap N)$
turning $p:=q|_U\colon U\to U/(U\cap N)$
into a smooth submersion,
then $U/(U\cap N)$
is a Lie group and there exists
a unique Lie group structure on $G/N$
which makes $U/(U\cap N)$
an open submanifold. The latter makes
$q\colon G\to G/N$ a smooth submersion.
\end{la}
\begin{proof}
The map $p$ is a surjective smooth submersion,
whence also the map $p\times p\colon U\times U\to U/(U\cap N)\times U/(U\cap N)$
is a surjective smooth submersion.
Let $\eta_U\colon U\to U$
and $\eta\colon U/(U\cap N)\to U/(U\cap N)$
be the mappings taking a group element to its inverse;
let $m_U\colon U\times U\to U$
and $m\colon U/(U\cap N)\times U/(U\cap N)\to U/(U\cap N)$
be the group multiplication.
Since $\eta\cap p=p\cap \eta_U$ is smooth
and $p$ is a surjective smooth submersion,
$\eta$ is smooth (see \cite[Exercise~1.7.6]{Sch23}).
Likewise, the smoothness of $m\circ (p\times p)=p\circ m_U$
implies that $m$ is smooth. Hence $U/(U\cap N)$
is a Lie group. Since $q$ is an open map,
$q(U)$ is open in $G/N$.
For $g\in G$,
the inner automorphism $G\to G$, $x\mto gxg^{-1}$
is smooth. Since $U$ is open,
we find an open identity neighbourhood $V\sub U$
such that $\alpha_g(V)\sub U$.
Consider the inner automorphism $\beta_g\colon G/N\to G/N$,
$xN\mto (gN)(xN)(gN)^{-1}$.
Then
\[
\beta_g\circ q=q\circ\alpha_g,
\]
whence $\beta_g(q(V))=q(\alpha_g(V))\sub q(U)$.
Here $q(V)=p(V)$ and $q(U)=U/(U\cap N)$.
By the preceding,
$\beta_g$ restricts to a map
$\beta_g|_{p(V)}\colon p(V)\to U/(U\cap N)$.
Since $p|_V\colon V\to p(V)$
is a surjective $C^\infty$-submersion and
$\beta_g|_{p(V)}\circ p|_V=p\circ \alpha_g|_V$
is smooth, also $\beta_g|_{p(V)}$
is smooth.
Hence $G/N$ has a unique
smooth manifold structure
making $U/(U\cap N)$ an open submanifold,
by the local description of
Lie group structures in
\cite[Proposition 1.13]{GCX}
(analogous to \cite[Chapter~III, \S1, no.\,9 Proposition~18]{Bou}).
Given $q\in G$, write $R_g\colon G\to G$, $x\mto gx$
and $R_{gN}\colon G/N\to G/N$, $xN\mto xNgN$. 
Since $q|_U$ is a submersion,
we deduce that
\[
R_{(gN)^{-1}}\circ q|_{gU}=
q|_U\circ R_{g^{-1}}|_{gU}
\]
is a submersion for each $g\in G$,
whence $q$ is a submersion.
\end{proof}
\noindent
{\bf Proof of Corollary~\ref{new-cor}.}
We know that $\Diff^{\fr}(M)$
is an open subgroup of the Lie group
$\Diff(M)$ (cf.\ \cite{Glo23}).
As each diffeomorphism of~$M$
leaves $\partial M$ invariant,
$\Diff^{\partial,\id}(M)$
is a normal subgroup of $\Diff(M)$.
Moreover, $\Diff^{\partial,\id}(M)$
is closed in $\Diff(M)$
as the point evaluations $\Diff(M)\to M$, $\psi\mto\psi(x)$
are continuous for all $x\in\partial M$.
Moreover, $\Diff^{\fr}(M)/\Diff^{\partial,\id}(M)$
admits a Lie group structure turning
the canonical quotient map
into a smooth submersion (see Remark~\ref{post-thm-B}).
Thus all hypotheses of Lemma~\ref{from-open}
are satisfied. $\,\square$
\section{Regularity of quotient groups}
We recall regularity properties
of Lie groups
and record an observation
concerning regularity of quotient
groups, which can then be applied
to quotients of
diffeomorphism groups of polytopes.
\begin{numba}
Let $G$ be a Lie group modeled on a locally
convex space, with neutral element~$e$
and Lie algebra $\cg:=L(G)$.
For $g\in G$, the right translation
$R_g\colon G\to G$, $x\mto xg$
is smooth. We get a right action of
$G$ on its tangent bundle $TG$ via
\[
TG\times G\to TG,\quad (v,g)\mto TR_g(v)=: v.g.
\]
If $\eta\colon [0,1]\to G$ is a $C^1$-curve,
let
\[
\delta(\eta)\colon [0,1]\to\cg,\quad t\mto \dot{\eta}(t).\eta(t)^{-1}
\]
be its right logarithmic derivative.
For a continuous curve $\gamma\colon [0,1]\to\cg$,
there is at most one $C^1$-curve
$\eta\colon [0,1]\to G$ such that
\begin{equation}\label{def-evol}
\delta(\eta)=\gamma\quad\mbox{and}\quad \eta(0)=e
\end{equation}
(cf.\ \cite{Glo16,Nee06}).
If $\eta$ exists, it is called the \emph{evolution} of~$\gamma$
and we write $\Evol(\gamma):=\eta$.
Endow $C([0,1],G)$
with its natural smooth Lie group
structure modeled on $C([0,1],\cg)$.
\end{numba}
\begin{numba}
Let $k\in \N_0\cup\{\infty\}$.
The Lie group~$G$ is called \emph{$C^k$-semiregular}
if $\Evol(\gamma)$ exists for each $\gamma\in C^k([0,1],\cg)$.
If, moreover, $\Evol\colon C^k([0,1],\cg)\to C([0,1], G)$
is smooth, then $G$ is called \emph{$C^k$-regular}
(cf.\ \cite{Glo16}).
It is clear from the definition that $C^k$-regularity implies
$C^\ell$-regularity for all $\ell\geq k$.
Thus $C^\infty$-regularity (also simply called \emph{regularity})
is the weakest condition. For Lie groups with sequentially
complete modeling spaces, the concept goes back to~\cite{Mil84}.
For $p\in \{p\}\cup [1,\infty[$
a Lie group $G$
with sequentially complete modeling space
is called \emph{$L^p$-semiregular}
if an evolution $\Evol(\gamma)$
exists for each $\gamma\in L^p([0,1],\cg)$,
looking now for Carath\'{e}odory
solutions $\eta\colon [0,1]\to G$ to (\ref{def-evol})
which are merely absolutely
continuous. If, moreover, $\Evol\colon L^p([0,1],\cg)\to C([0,1],G)$
is smooth, then $G$ is called \emph{$L^p$-regular}.
If $G$ is $L^p$-regular, then $G$ is $L^q$-regular for all $q\geq p$
and $C^0$-regular. See \cite{Glo15,Nik21,GH23} for details.
\end{numba}
\noindent
Notably, each $L^1$-regular Lie group~$G$
is $C^0$-regular.\footnote{The evolution map
$\Evol\colon L^1([0,1],L(G))\to C([0,1],G)$
restricts to a smooth map $C([0,1],L(G))\to C([0,1],G)$
which is the evolution map on continuous curves.}
\begin{la}\label{cheap-reg}
Let $\alpha\colon G\to H$
be a smooth group homomorphism
between Lie groups modeled
on locally convex spaces
such that
\[
L(\alpha)\colon L(G)\to L(H)
\]
admits a continuous linear right inverse.
Then the following holds:
\begin{itemize}
\item[\rm(a)]
If $k\in \N_0\cup\{\infty\}$
and $G$ is $C^k$-semiregular,
then also $H$ is $C^k$-semiregular.
\item[\rm(b)]
If $k\in \N_0\cup\{\infty\}$
and $G$ is $C^k$-regular,
then also $H$ is $C^k$-regular.
\item[\rm(c)]
If $G$ and $H$ are modeled
on sequentially complete
locally convex spaces, $p\in \{\infty\}\cup [1,\infty[$
and $G$ is $L^p$-semiregular,
then also $H$ is $L^p$-semiregular.
\item[\rm(d)]
If $G$ and $H$ are modeled
on sequentially complete
locally convex spaces, $p\in \{\infty\}\cup [1,\infty[$
and $G$ is $L^p$-regular,
then also $H$ is $L^p$-regular.
\end{itemize}
\end{la}
\begin{proof}
Let $\sigma\colon L(H)\to L(G)$
be a continuous linear right inverse
for $L(\alpha)$. (a) and (b): The map
\[
\sigma_*\colon C^k([0,1],L(H))\to C^k([0,1],L(G)),\quad \gamma\mto\sigma\circ\gamma
\]
is continuous and linear.
If $G$ is $C^k$-semiregular
and $\gamma\in C^k([0,1],L(G))$, then $\sigma\circ\gamma\in C^k([0,1],L(H))$
has an evolution $\eta:=\Evol_G(\sigma\circ \gamma)\colon [0,1]\to G$.
The left logarithmic derivative of $\alpha\circ\eta$ is
\[
\delta(\alpha\circ \eta)=L(\alpha)\circ \delta(\eta)=L(\alpha)\circ \sigma\circ \gamma=
\gamma
\]
(cf.\ \cite[Proposition~II.4.1\,(1)]{Nee06}).
Moreover, $(\alpha\circ\eta)(0)=\alpha(\eta(0))=\alpha(e_G)=e_H$.
Thus $\alpha\circ\eta=\Evol_H(\gamma)$ and
\[
\alpha_*\circ \Evol_G\circ \, \sigma_*=\Evol_H,
\]
using the mapping $\alpha_*\colon C([0,1],G)\to C([0,1],H)$,
$\zeta\mto\alpha\circ \zeta$ which is smooth
(see, e.g., \cite[Corollary 1.22]{AGS20}).
If $G$ is $C^k$-regular, then $\Evol_G$
is smooth and hence also $\Evol_H=\alpha_*\circ \Evol_G\circ \sigma_*$.\\[2.3mm]
The proof of (c) and (d) is analogous,
as $\sigma_*\colon L^p([0,1],L(H))\to L^p([0,1],L(G))$,
$[\gamma]\mto [\sigma\circ\gamma]$
is continuous linear. We need only replace the symbol
$C^k$ with $L^p$
and $C^k$-functions with equivalence
classes of $\cL^p$-functions.
\end{proof}
\begin{prop}\label{is-regular}
In the situation of Theorem~{\rm\ref{thm-A}},\vspace{.6mm}
the Lie group $\im(\rho)\cong$\linebreak
$\Diff^{\fr}(M)/\ker(\rho)$
is $L^1$-regular and hence $C^k$-regular for each
$k\in \N_0\cup\{\infty\}$.
\end{prop}
\begin{proof}
The map $\rho\colon \Diff^{\fr}(M)\to\im(\rho)$
discussed in Theorem~\ref{thm-A} is a smooth
group homomorphism between
Lie groups and a submersion.
Hence $L(\rho)=T_e(\rho)$
has a continuous linear right inverse
(see \cite[ 1.56]{Sch23}).
Since
$\Diff^{\fr}(M)$
is $L^1$-regular (cf.\ \cite[Remark 1.3]{Glo23}),
Lemma~\ref{cheap-reg}\,(d)
shows that $\im(\rho)$ is $L^1$-regular
and thus $C^k$-regular for all $k\in \N_0\cup\{\infty\}$.
\end{proof}

\section{Controllability on simple polytopes}\label{sect:control}

We will now consider sufficient conditions for when we can generate all diffeomorphisms with compositions of flows of vector fields. Our approach will be to show controllability locally and then combine local results together, taking advantage of the fact that $M$ is compact.

As flows of vector fields on polytopes are non-standard, we shall first establish their properties from the regularity of the Lie group $\Diff^{\fr}(M)$:

\begin{numba}
Let $M$ be a polytope with dense interior embedded in $\R^n$. Following the construction of the regular Lie group $\Diff^{\fr}(M)$ in \cite{Glo23}, the group $\Diff^{\fr}(M)$ is an open subset of the affine subspace $\id_M + C^\infty_{\mathrm{str}} (M,\R^n)$, hence a submanifold of $C^\infty (M,\R^n)$. Let $\gamma \in C^k([0,1],T_{\id}\Diff^{\fr}(M))$ be a curve into the tangent space at the identity. Since $\Diff^{\fr}(M)$ is $C^k$-regular, \cite{Glo23}, its right-evolution $\eta := \Evol(\gamma)$ exists and solves the differential equation $\dot{\eta} = T_{\id} R_{\eta} (\gamma)$.
We will now identify the right evolution of $\gamma$.

The standard calculations for tangent spaces of manifolds of mappings (see e.g. \cite[Appendix C]{Sch23} or \cite[Appendix A]{AGS20}) yield for $g \in \Diff^{\fr}(M)$
$$T_g \Diff^{\fr}(M) = T_g C^\infty_{\text{str}} (M,\R^n) = \{X \circ g\mid X \in C^\infty_{\text{str}}(M,\R^n)\} \cong \{g\} \times C^\infty_{\text{str}} (M,\R^n).$$
Indeed, the identification of the tangent space is induced by the continuous linear right translation $r_g \colon C^\infty (M,\R^n) \rightarrow C^\infty (M,\R^n),\ f\mapsto f \circ g$. For $g\in \Diff^{\fr}(M)$ this map restricts to the right multiplication $R_g \colon \Diff^{\fr}(M) \rightarrow \Diff^{\fr}(P)$, whence the tangent maps of $R_g$ are 
$$T_h R_g \colon \{h\} \times C^\infty_{\text{str}} (M,\R^n) \rightarrow \{h \circ g\} \times C^\infty_{\text{str}} (M,\R^n), \quad T_h R_g (h,f)=(h\circ g, f \circ g).$$
Identify $T_{\id} \Diff^{\fr}(M)\cong C^\infty_{\text{str}}(M,\R^n)$ to obtain a formula for the right evolution of a curve $\gamma \in C^k([0,1],C^\infty_{\textrm{str}}(M,\R^n))$ viewed as a time dependent vector field on $M$. Recall that the point evaluation $\ev_x \colon C^\infty_{\rm str}(M,\R^n) \rightarrow \R^n, g \mapsto  g(x)$ is continuous linear for every $x\in M$. Hence, we obtain for each $x\in M$,
$$\dot{\eta}(t)(x)=\ev_x (\dot{\eta}(t))=\ev_x \left(T_{\id} R_{\eta(t)}(\gamma(t))\right) = \gamma(t)(\eta(t)(x)).$$
In other words we see that $\Evol(\gamma)(t)(x)=\Fl_t^\gamma(x)$ for every $x\in M$, where $\Fl_t^\gamma$ is the flow of the (time dependent) vector field $\gamma$. Hence the evolution of the curve $\gamma$ coincides with the flow of the vector field $\gamma$.
In particular, flows of stratified vector fields give rise to face respecting diffeomorphisms.
 \end{numba}
 Before we recall a result from \cite[Proposition~4.1]{AgCa09} let us define the following notation: For any neighbourhood $V$ of zero, we write $C^\infty(V,\mathbb{R}^n;0)$ for the space of smooth functions $V \rightarrow \R^n$ mapping zero to itself.
\begin{la} \label{lem:Xn}
If $X_1,\dots, X_n$ are vector fields on $\mathbb{R}^n$ such that
$$\spn \{ X_1(0), \dots, X_n(0) \} = T_0 \mathbb{R}^n,$$
then there exists a relatively compact neighbourhood $U$ of $0$ and a neighbourhood $\mathscr{U}$ of $C^\infty(U,\mathbb{R}^n;0)$, such that any $F \in \mathscr{U}$ can be written as
$$\psi = e^{f_1 X_1} \circ e^{f_2 X_2} \circ \cdots \circ e^{f_n X_n}|_U, \qquad f_j \in C^\infty(\mathbb{R}^n), \qquad f_j(0) = 0.$$
\end{la}

Recall the following result \cite[Lemma 3.4]{GS24}, where in the formulation $(y_1, \dots, y_n)$ denotes the standard coordinates on $\R^n$.
\begin{la} \label{lemma:varphi}
Let $\hat Z$ be a vector field on $\R^n$ such that $dy_n(\hat Z)(0) \neq 0$. Then there exists a function $u: \R^n \to \R$ of compact support such that
\begin{enumerate}[\rm (a)]
\item $u|_V =1$ for some neighbourhood $V$ of $0$.
\item If $\hat g \in C^\infty(\mathbb{R}^n)$ is a function such that $e^{\hat gu \hat Z}$ preserves $\{ y_n =0\} \cap V$, then on the support of $u$ we can write $\hat g(x) = y_n g(x)$ for some $g \in C(\R^n)$.
\end{enumerate}
\end{la}
Finally, a version of Seeley`s extension argument for differentiable mappings on simple polytopes, a special case of \cite[Application 1]{Han23}, will be needed: 
\begin{la} \label{lem:Seely}
Consider $Q_0= [0,1[^i \times ]-1,1[^{n-i}\subseteq \R^n$, $\hat U =]-\infty,1[^i \times ]-1,1[^{n-i} \subseteq \mathbb{R}^n$ and a locally convex space $F$. Then for each $k\in \N\cup \{\infty\} $ there exists a continuous linear extension operator 
$$\Ext^k \colon  C^k (Q_0,F) \rightarrow C^k (\hat U,F).$$
\end{la}

We are now in a position to generalise the local controllability results obtained in \cite{GS24} for manifolds with smooth boundary. For this we will first establish a localisation result on a simple embedded polytope $M\subseteq \R^n$.
\begin{numba} \label{subsect:loc_lem}
Let $x_0$ be any arbitrary point in $M$. By using local coordinates $(y_1, \dots, y_n)$ as in Lemma~\ref{sta-cha}, we many assume that $x_0 =0$ in $Q_0 \colon= [0,1 [^i \times ]-1,1[^{n-i}$ contained in some larger cube $Q \subseteq \mathbb{R}^n$ with $y_j =0$ as a facet for $j=1, \dots,i$. 

Assume that we have space of vector fields $\scrV$ on $Q$ satisfying the assumptions of Theorem~\ref{thm-control}. Write $\hat F_j$ for the facet $y_j =0$.
Define the face $F = \cap_{j=1}^i \hat F_j$, and observe that $0$ is in the interior of this face. If $x_0$ is in the interior of $Q$, then we put $i=0$ and $F = Q$. Since $\Gr(\scrV)$ acts transitively on the interior of $F$, by Sussmann's Orbit Theorem, see \cite{McKay07}, there exists vector fields $Y_1,\dots,Y_{n-i}$ and $\varphi_1,\dots, \varphi_{n-i} \in \Gr(\scrV)$ such that if $X_j = \varphi_{j,*} Y_j$, then
$$T_0 F = \spn\{ X_1, \dots, X_{n-i}\}.$$
If $i=0$, we have a basis for $T_0M$. For the remaining cases, we use assumption (II) of Theorem~\ref{thm-control} and the following argument to construct a full basis.

For we know that for any $j=1, \dots,i$, there exists a point $x_j$ with neighbourhood $V_j$ and vector field $Z_j$ such that $Z_j|_{V_j \cap \hat F_j}  = 0$ and with $\nabla_{\nu_j} Z_j \neq T\hat F_j$. By choosing a local coordinate system $(y_{j,1},\dots,y_{j,n})$ in $V_j$ with $x_j$ in center, such that $\hat F_{j_2} \cap V_j$ correspond to $y_{j,j_2} =0$. Then we can write $Z_j$ as $Z_j =\sum_{l =1}^n Z^l_j \partial_{y_{j,l}}$ with $\partial_{y_{j,j}} Z^l_j(x_j) \neq 0$. 
Without loss of generality, we can rescale $Z_j$ to have $\partial_j Z_ j^j =1$. Note since $Z_j$ is $\Vect_{\rm str}(M)$, it also has to be tangent to $\hat F_l$ for $l \neq j$, which means that $Z^l_j(\hat F_l \cap V_j) =0$, and in particular implies that $\partial_{y_{j,j}} Z^l(\hat F_l \cap V_j)=0$ as $\partial_{y_{j,j}}$ is tangent to $\hat F_l$. It now follows since $0 \in F = \cap_{l=1}^i \hat F_l$ that
$$\sum_{l=1}^n \partial_{y_{j,j}} Z^l_j(0) \partial_{y_{j,l}}= \partial_{y_{j,j}} \mod T_0 F.$$
In summary we can locally write  $Z_j = y_{j,j} \hat Z_j$ with $\hat Z_j(0) = \partial_{y_{j,j}} \bmod T_0 F$.
Finally, let $u_j$ be a bump function on $V_j$ satisfying Lemma~\ref{lemma:varphi}, and let $\hat \varphi_j \in \scrG$ be a map satisfying $\varphi_j(x_j) = x_0$, which exists by (I). Define local vector fields around $x_0$,
$$\hat X_j =\varphi_{j,*} u_j Z_j.$$
We have
$$T_0 M = \spn \{ X_1, \dots, X_{n-i}, \hat X_1, \dots, \hat X_i\}$$
where we have identified $x_0$ with $0$ in the coordinates $(y_1,\dots, y_n)$.

Define a larger cube $\hat Q = \prod_{j=1}^n [-\sup_{y\in Q} |y_j|,  \sup_{y\in Q} |y_j|]$. By using the Lemma~\ref{lem:Seely}, we can extend the vector fields to $\hat Q$.
We need another space of structure preserving maps which are in spirit similar to the stratified vector fields of \Cref{facesandstratvec}.

\begin{defn}
For a any open set $U$ in a simple embedded polytope $Q\subseteq \R^n$, write $C^\infty_{\rm fr}(U,Q)$ for the smooth functions $\phi$ such that $\phi(F) \subseteq F$ for any face $F$. 
We endow $C^\infty_{\rm fr}(U,Q)$ with the compact open $C^\infty$-topology (i.e.  the subspace topology induced by the inclusion $C^\infty_{\rm fr}(U,Q) \subseteq C^\infty(U,Q)$.
\end{defn}
 In other words elements in $C^\infty_{\rm fr} (U,Q)$ are face respecting. Note that the group $\Diff^{\rm fr}(Q)$ is contained by definition in $C^\infty_{\rm fr}(Q,Q)$.
 Let now $\phi \in C^\infty_{\rm fr}(Q_0,Q)$. By definition, we must have $\phi(0) \in \intOP F$ and since $\Gr(\scrV)$ acts transitively on the interior of $F$, we can write $\phi = \phi_1 \circ \psi$ where $\phi_1 \in \Gr(\scrV)$ and $\psi(0)=0$. Furthermore, by combining Lemma~\ref{lem:Xn} and Lemma~\ref{lemma:varphi}, there exists smooth functions $f_1,\dots, f_{n-i}$, $\hat g_1, \dots, \hat g_i$, $g_1, \dots, g_i$ and a neighbourhood $U_0$ of $0$ in $\R^n$ such that
\begin{align*}
\Ext^\infty(\psi)|_{U_0} &= e^{f_1 X_1} \circ \cdots \circ e^{f_{n-i}X_{n-i}} \circ e^{\hat g_1 \hat X_1}  \circ \cdots \circ e^{\hat g_i \hat X_i}|_{U_0} \\
&= \varphi_1 \circ e^{(f_1 \circ \varphi_1) Y_1} \circ \varphi_1^{-1} \circ \cdots \circ \varphi_{n-i} \circ  e^{(f_{n-i} \circ \varphi_{n-i}) Y_{n-i}} \varphi_{n-i}^{-1} \\
& \qquad \circ \hat \varphi_1 \circ  e^{(\hat g_1 \circ \hat \varphi_1) u_1 \hat Z_i} \hat \varphi_i^{-1} \circ \cdots \circ \hat \varphi_i \circ  e^{(\hat g_i \circ \hat \varphi_i) u_i \hat Z_i} \circ \hat  \varphi_i^{-1}|_{U_0} \\
&= \varphi_1 \circ e^{(f_1 \circ \varphi_1) Y_1} \circ \varphi_1^{-1} \circ \cdots \circ \varphi_{n-i} \circ  e^{(f_{n-i} \circ \varphi_{n-i}) Y_{n-i}} \varphi_{n-i}^{-1} \\
& \qquad \circ \hat \varphi_1 \circ  e^{( g_1 \circ \hat \varphi_1) u_1 Z_i} \hat \varphi_i^{-1} \circ \cdots \circ \hat \varphi_i \circ  e^{( g_i \circ \hat \varphi_i) u_i Z_i} \circ \hat  \varphi_i^{-1}|_{U_0}
\end{align*}
with $\hat g_j \circ \hat \varphi_j = y_{j,j} g_j \circ \varphi_j$ and with $U_0$ chosen independently of $\psi$. Restricting back to $Q_0$ and defining $U = U_0 \cap Q_0$, we have our desired result that for some neighbourhood $\scrU$ of the identity in $C^\infty_{\rm fr}(U, Q)$, we have $\scrU \subseteq \Gr(\hat \scrV)$. Since this result is local, we obtain the following general conclusion.
\begin{la}[Localisation lemma] \label{lemma:loc}
Let $M$ be a simple polytope and $x \in M$. Then there exists a neighbourhood $U$ of $x$ and a neighbourhood $\scrU$ of the identity in $C_{\mathrm{fr}}^\infty(U,M)$ such that $\scrU \subseteq \Gr(\hat \scrV)|_U$.
\end{la}
\end{numba}

 We are now ready to prove our main results for this section.
\begin{proof}[Proof of Theorem~\ref{thm-identity-boundary}]
We can choose an open cover of $M$, by neighbourhoods satisfying the properties in the Localisation \Cref{lemma:loc}. By compactness of $M$ we may select a finite subcover $U_i, i=1,\ldots ,k$ together with a subordinate partition of unity $\lambda_i, i=1,\ldots , k$. Then we can prove the result using standard fragmentation techniques as in \cite[Section~3.2]{GS24}. For the readers convenience we repeat the key steps: Let $\mathscr{O}$ be an open, connected $\id_M$-neighbourhood in $\Diff^{\fr} (M)$. Define $\mathscr{O}_j = \{\phi \in \mathscr{O} | \text{supp} \phi \subseteq U_i\}$ and note that $\mathscr{O}\subseteq \mathscr{O}_1 \circ \mathscr{O}_2 \circ \cdots \circ \mathscr{O}_k$. To see this pick for  $\phi \in \mathscr{O}$ a smooth curve $s \mapsto \varphi_s \in \mathscr{O}$ with $\varphi_0 = \id_M$ and $\varphi_1 = \phi$. Then define $\psi_j(x):= \varphi_{\lambda_1 (x)+ \cdots + \lambda_j (x)}(x)$ and $\phi_j$, $j=1,2, \dots, l$, by
$$\phi_1 = \psi_1, \qquad \phi_{j+1} = \psi_{j+1} \circ \psi_j^{-1}.$$
Then $\phi_{j} \in \mathscr{O}_j$. Using the open identity neighbourhoods $\mathscr{U}_j$ from \Cref{lemma:loc} define a new neighbourhood $\hat{\mathscr{O}}$ generated by $\mathscr{O}_j \cap \mathscr{U}_j$. Since $\hat{\mathscr{O}}\subseteq \hat{\mathscr{G}}$, the result follows.
\end{proof}

The proof of the following result is similar to \cite[Corollary 1.3]{GS24}.
\begin{proof}[Proof of Theorem~\ref{thm-control}]
 We choose $\scrV = \hat{\scrV} = \Vect_{\rm str}(M)$ and recall the steps in the beginning of \Cref{subsect:loc_lem} leading to the proof of the Localisation lemma, Lemma~\ref{lemma:loc}. In particular, we adhere to the notation introduced there. If $x_0$ is then a given point on the boundary, we can choose local coordinates $y$  but for each point on the boundary, we choose local coordinates $(y_1, \dots, y_n)$, such that if $x_0$ is contained exactly in the facets $\hat F_1$, \dots, $\hat F_i$ and no other, then these are on the respective hyperplanes $y_1 = 0$, \dots, $y_i = 0$. Define vector fields 
$$Z_j = y_j \hat Z_j = y_j \partial_{y_j}, \qquad j= 1,\dots, i,$$
and $X_j = \partial_{y_j}$ for $j = i+1,\dots, n$. It follows then by the proof of Lemma~\ref{lemma:loc}, that we can write any diffeomorphism preserving in some neighbourhood $U$ as
$$\phi|_U = e^{f_{i+1} X_{i+1}} \circ \cdots \circ e^{f_n X_n} \circ e^{\hat g_1 \hat Z_1} \circ \cdots \circ e^{\hat g_i \hat Z_i}|_U,$$
However, since $\phi$ is the identity on the boundary, the same properties have to hold for all of these flows in separate coordinates, giving us that
$$\phi = e^{f_{i+1} X_{i+1}} \circ \cdots \circ e^{f_n X_n} \circ e^{g_1 Z_1} \circ \cdots \circ e^{g_i Z_i}, \qquad f_j|_{\cap_{r=1}^i \hat F_r} =0, j= i+1, \dots n,$$
and with $\hat g_j = y_j g_j$ for $j=1,\dots, i$.
This shows that locally around any point on the boundary, any diffeomorphism that is the identity on the boundary, can be written as a composition of flows of vector field in $\Vect^{\partial=0}(M)$. Finally, we now follow the steps of \cite[Proof of 1.2 in Section 3.2]{GS24} to construct the neighbrohood of the identity from the local open neighborhoods just constructed. This finishes the proof.
\end{proof}

%% Square example
\begin{example}
We return to the square $S= [0,1]^2$ discussed in \Cref{ex:square}. The square and its diffeomorphism group have recently been considered as the most basic example for applications in numerical analysis. See e.g. \cite{CaGaRaS23} for an account of machine learning techniques on diffeomorphism groups of the square and other polytopes. Our results show that any diffeomorphism of the square can be generated by vector fields that are tangent to the boundary. Furthermore, we can use vector fields that vanish at the boundary to generate diffeomorphisms that equals the identity on the boundary. But we can also use smaller family of vector fields such that
$$Z_1 = (2y-1) x(1-x) \partial_x, \qquad Z_2 = y(y-1) \partial_y.$$
If we define $\scrV= \spn \{ Z_1, Z_2\}$, then this collection satisfies the assumptions of Theorem~\ref{thm-control}, even though the vector fields only span a one-dimensional space along the line $y = 1/2$.
\end{example}

\printbibliography
\noindent
{\small{\bf Helge Gl\"{o}ckner}, Universit\"{a}t Paderborn, Department of
Mathematics,\\
Warburger Str.\ 100,
33098 Paderborn,
Germany; glockner@math.uni-paderborn.de\\[2.3mm]
{\bf Erlend Grong},
University of Bergen, Department of Mathematics, P.O.\ Box 7803, 5020 Bergen,
Norway;
erlend.grong@uib.no\\[2.3mm]
{\bf Alexander Schmeding},
Institutt for matematiske fag, NTNU Trondheim,\\
Alfred Getz' vei 1 Gl\o{}shaugen,
7034 Trondheim, Norway;
alexander.schmeding@ntnu.no}\vfill
\end{document}